\documentclass[a4paper,11pt]{amsart}
\usepackage{amsmath,amsthm,amssymb}
\usepackage[mathscr]{eucal}
 \usepackage[noadjust]{cite}
\usepackage{upgreek}
\usepackage[bookmarks,bookmarksnumbered]{hyperref}
\usepackage{enumerate}
\usepackage{amsmath}
\usepackage{mathrsfs}
\usepackage{blindtext}
\usepackage{scrextend}
\usepackage{enumitem}
\usepackage{bm} 
\addtokomafont{labelinglabel}{\sffamily}
\usepackage{color}
\usepackage[at]{easylist}

\numberwithin{equation}{section}

\theoremstyle{plain}
\newtheorem{main}{Theorem}
\newtheorem{mcor}[main]{Corollary}

\newtheorem{theorem}{Theorem}[section]
\newtheorem{claim}[theorem]{Claim}
\newtheorem{lemma}[theorem]{Lemma}
\newtheorem{proposition}[theorem]{Proposition}

\theoremstyle{definition}
\newtheorem{definition}[theorem]{Definition}
\newtheorem*{definition*}{Definition}

\newtheorem{remark}[theorem]{Remark}
\newtheorem{fact}[theorem]{Fact}

\newcommand{\F}{\mathbb{F}}

\newcommand{\cH}{\mathcal{H}}

\newcommand{\Tr}{\operatorname{Tr}}

\begin{document}

\title[Existential closedeness and the structure of bimodules of II$_1$ factors]
{Existential closedeness and the structure \\ of bimodules of II$_1$ factors}

\author[A. Ioana]{Adrian Ioana}
\address{Department of Mathematics, University of California San Diego, 9500 Gilman Drive, La Jolla, CA 92093, USA}\email{aioana@ucsd.edu}

\author[H. Tan]{Hui Tan}

\address{Department of Mathematics, University of California San Diego, 9500 Gilman Drive, La Jolla, CA 92093, USA}\email{hutan@ucsd.edu}

{\thanks{A.I. was partially supported by NSF DMS grants 1854074 and 2153805, and a Simons Fellowship.}}
{\thanks{H.T. was partially supported by NSF DMS grants 1854074 and 2153805.}}
\begin{abstract} 
We prove that if a separable II$_1$ factor $M$ is existentially closed, then every $M$-bimodule is weakly contained in the trivial $M$-bimodule, $\text{L}^2(M)$,  and, equivalently, every normal completely positive map on $M$ is a pointwise 2-norm limit of maps of the form $x\mapsto\sum_{i=1}^ka_i^*xa_i$, for some $k\in\mathbb N$ and $(a_i)_{i=1}^k\subset M$.
This provides the first examples of non-hyperfinite separable II$_1$ factors $M$ with the latter properties.
We also 
obtain new characterizations of $M$-bimodules which are weakly contained in the trivial or coarse $M$-bimodule and of relative amenability inside $M$.
Additionally, we give an operator algebraic presentation of the proof of the existence of existentially closed II$_1$ factors. 
While existentially closed II$_1$ factors  have property Gamma, by adapting this proof we construct non-Gamma II$_1$ factors which are existentially closed in every weakly coarse extension. 
 \end{abstract}

\maketitle

\section{Introduction and statement of main results}
In this paper, motivated by the notion of existentially closed II$_1$ factors in continuous model theory, we investigate the structure of bimodules of II$_1$ factors. The model theoretic study of II$_1$ factors was initiated by Farah, Hart and Sherman in \cite{FHS11} (see \cite{GH22} for a recent survey). Existentially closed II$_1$ factors were first considered in \cite{GHS12,FGHS13} and have since been studied in \cite{GHS12,FGHS13,Go18, AGKE20, GH21,CDI22a,CDI22b,GJKEP23}, while bimodules of II$_1$ factors  were studied from the perspective of model theory in \cite{GHS18}.
Existentially closed II$_1$ were shown to be McDuff in \cite{GHS12} and to only have approximately inner automorphisms in \cite[Proposition 3.1]{FGHS13}.  
 The original, model theoretic definition of existentially closed II$_1$ factors involves existential formulae, see \cite[Definition 1.1]{FGHS13}. In this paper, we will restrict to separable II$_1$ factors and work with the following equivalent operator algebraic definition (see, e.g., \cite[Remarks 3.2]{GH21} and the comments following \cite[Lemma 5.2]{GH22}).


\begin{definition}\label{ecdef} If $M\subset N$ are separable II$_1$ factors, we say that $M$ is {\it existentially closed} in $N$  if there is a $*$-homomorphism $\pi:N\rightarrow M^{\mathcal U}$, for a free ultrafilter $\mathcal U$ on $\mathbb N$, such that the restriction of $\pi$ to $M$ is the diagonal embedding of $M$ into $M^{\mathcal U}$.
A separable II$_1$ factor $M$ is called {\it existentially closed} if 
it is existentially closed in any separable II$_1$ factor $N$ containing it.
\end{definition}


Our first main result shows that being existentially closed has strong consequences to the structure of bimodules and completely positive maps of a II$_1$ factor.

\begin{main}\label{ec} 
Consider the following conditions for a separable II$_1$ factor $M$.
\begin{enumerate}
\item $M$ is existentially closed.
\item Every 
$M$-bimodule $\mathcal H$ satisfies $\mathcal H\subset \emph{L}^2(M^{\mathcal U})$, for some ultrafilter $\mathcal U$ on a set $I$.  Moreover, if $\mathcal H$ is separable, then we can take $I=\mathbb N$.
\item Every 
$M$-bimodule $\mathcal H$ belongs to the closure of $\emph{L}^2(M)$, in the Fell topology.
\item Every $M$-bimodule $\mathcal H$ is weakly contained in $\emph{L}^2(M)$.
\item For every normal, completely positive map $\Phi:M\rightarrow M$, there exists a sequence $(\Phi_n)\subset\mathcal P_M$ such that $\lim\limits_{n\rightarrow\infty}\|\Phi_n(x)-\Phi(x)\|_2=0$, for every $x\in M$.
Moreover, if $\Phi$ is subunital and subtracial, then we can take $(\Phi_n)\subset\mathcal S_M$.
\end{enumerate} 
Then (1) $\Rightarrow$ (2) $\Rightarrow$ (3) $\Rightarrow$ (4) $\Leftrightarrow$ (5), for every separable II$_1$ factor $M$.
\end{main}

Here, we denote by $\mathcal P_M$ the set of maps $\Phi:M\rightarrow M$ of the form $\Phi(x)=\sum_{i=1}^ka_i^*xa_i$, for some  $(a_i)_{i=1}^k\subset M$, and by $\mathcal S_M$ the set of maps $\Phi:M\rightarrow M$ of the form $\Phi(x)=\sum_{i=1}^ka_i^*xa_i$, for some  $(a_i)_{i=1}^k\subset M$ with $\sum_{i=1}^ka_i^*a_i\leq 1$ and  $\sum_{i=1}^ka_ia_i^*\leq 1$. Every  $\Phi\in\mathcal P_M$ is normal and completely positive, and every  $\Phi\in\mathcal S_M$ is normal, subunital, subtracial and completely positive. For the other terminology used in the statement of Theorem \ref{ec}, we refer to Section \ref{prelim}.

Before commenting on the proof of Theorem \ref{ec}, let us highlight
 that it leads to a new phenomenon for II$_1$ factors. Let $\mathcal A$ be the family of separable II$_1$ factors $M$ which satisfy the equivalent conditions (4) and (5) from Theorem \ref{ec}. 
If $M$ is a separable II$_1$ factor, then every $M$-bimodule weakly contains the coarse bimodule, $\text{L}^2(M)\otimes\text{L}^2(M)$ (see \cite[Proposition 2.3.2]{Po86}). On the other hand, there exists a (unique up to weak equivalence) separable $M$-bimodule, denoted $\mathcal H_{\text{max}}$, which weakly contains every $M$-bimodule. Class $\mathcal A$ consists of II$_1$ factors $M$ for which  $\mathcal H_{\text{max}}=\text{L}^2(M)$. 
If $R$ is the hyperfinite II$_1$ factor \cite{MvN43}, then any two $R$-bimodules are weakly equivalent  (see \cite[Proposition 3.1.4]{Po86}), and therefore $R\in\mathcal A$. Also,  $\mathcal A$ is closed under amplifications and inductive limits, and thus we have $M\overline{\otimes}R\in\mathcal A$, for every $M\in\mathcal A$ (see Proposition \ref{classA}).
Theorem \ref{ec} provides the first examples of non-hyperfinite separable II$_1$ factors $M\in\mathcal A$.
In fact, it implies that $\mathcal A$ contains uncountably many non-isomorphic such II$_1$ factors, since there are uncountably many non-isomorphic separable existentially closed II$_1$ factors (see \cite[Corollary 1.3]{FGHS13}).

To prove that (1) implies (2) in Theorem \ref{ec}, we use Shlyakhtenko's  $M$-valued semicircular systems \cite{Sh97}, see Section  \ref{M-valued} for the definition. This construction, which associates a tracial von Neumann algebra  $\Gamma(M,\mathcal H)''$ containing $M$ to any symmetric $M$-bimodule $\mathcal H$, allows us to realize every $M$-bimodule  as a sub-bimodule of $\text{L}^2(N)$, for a tracial von Neumann algebra $N$ containing $M$.

\begin{remark}\label{q}
We do not know if conditions (2)-(5) in Theorem \ref{ec} are equivalent for an arbitrary separable II$_1$ factor $M$. If $M$ satisfies (2) or (3), then $M$ must have property Gamma. Indeed,  Connes' spectral gap theorem \cite{Co75} implies that  if the $M$-bimodule $\text{L}^2(M)\oplus\text{L}^2(M)$ belongs to the closure of $\text{L}^2(M)$, then $M$ has property Gamma, see \cite[Theorem 3.3.1]{Po86} and Lemma \ref{connes}. However, we do not know if the equivalent conditions (4) and (5)  imply property Gamma. 

If a separable II$_1$ factor $M$ satisfies conditions (4) and (5) (i.e., $M\in\mathcal A$ in the above notation) but fails property Gamma, then $M$ must have property (T) in the sense of \cite{Co80,CJ83}. 
In particular, the free group factors $\text{L}(\F_n)$, for $n\geq 2$, do not belong to $\mathcal A$, which answers a question asked by Jesse Peterson.
To justify our claim, let $\mathcal H$ be an $M$-bimodule which weakly contains $\text{L}^2(M)$. Since $M$ satisfies (4), $\mathcal H$ is weakly contained in $\text{L}^2(M)$. Since $M$ does not have property Gamma, \cite[Proposition 3.2]{BMO19} implies that $\mathcal H$ must contain $\text{L}^2(M)$ and thus $M$ has property (T).
\end{remark}

\begin{remark}\label{q'}
Partially addressing the question posed in  Remark \ref{q} in the first version of the paper, Amine Marrakchi proved that conditions (2)-(4) are equivalent for every separable II$_1$ factor $M$ with property Gamma. More precisely,  \cite[Corollary 3.9]{Ma23} shows that if $M$ has Gamma, then for every $M$-bimodule $\mathcal H$ we have $\mathcal H\subset\text{L}^2(M^{\mathcal U})$ if and only if  $\mathcal H$ belongs to the closure of $\text{L}^2(M)$ and if and only if $\mathcal H$ is weakly contained in $\text{L}^2(M)$. Since any separable II$_1$ factor satisfying (3)  must have property Gamma, it follows that (2) $\Leftrightarrow$ (3), for arbitrary separable II$_1$ factors $M$.

Let $\mathcal B\subset \mathcal A$ be the family of separable II$_1$ factors $M$ which satisfy the equivalent conditions (2) and (3) from Theorem \ref{ec}.
Then the previous paragraph implies that $\mathcal B$ consists of all $M\in\mathcal A$ with property Gamma.   By Remark \ref{q}, if $M\in\mathcal A\setminus \mathcal B$ (and thus $M$ does not have Gamma), then $M$ has property (T). It is an open question whether there are II$_1$ factors $M\in\mathcal A$ with property (T).
\end{remark}


Another natural question is whether conditions (1) and (2) from Theorem \ref{ec} are equivalent.
Our next result shows that (2) is satisfied by the hyperfinite II$_1$ factor $R$. By \cite[Corollary 2.2]{FGHS13}, $R$ is existentially closed if and only if the Connes Embedding Problem (CEP) has a positive answer.  A negative answer to the CEP has been announced in the preprint \cite{JNVWY20}.

\begin{mcor}\label{hyper}
Let $R$ be the hyperfinite II$_1$ factor. Then every $R$-bimodule $\mathcal H$ satisfies $\mathcal H\subset\emph{L}^2(R^{\mathcal U})$, for some ultrafilter $\mathcal U$ on a set $I$. Moreover, if $\mathcal H$ is separable, then we can take $I=\mathbb N$.
\end{mcor}

Corollary \ref{hyper} is a consequence of the following result on (approximate) embeddings of Shlyakhtenko's $M$-valued semicircular systems \cite{Sh97}. It can be alternatively deduced by using that $R$ satisfies condition (4) from Theorem \ref{ec} and that conditions (2) and (4) are equivalent for II$_1$ factors with property Gamma by \cite[Corollary 3.9]{Ma23}.

\begin{main}\label{bimod}
Let $(M,\tau)$ be a tracial von Neumann algebra and $\mathcal H,\mathcal K$ be two symmetric $M$-bimodules  such that $\mathcal H$ is weakly contained in $\mathcal K$.
Then there exists a trace preserving $*$-homomorphism $$\pi:\Gamma(M,\mathcal H)''\rightarrow (\Gamma(M,\mathcal K\otimes\ell^2S)'')^{\mathcal U},$$ for some ultrafilter $\mathcal U$ on a set $I$ and a set $S$, whose restriction to $M$ is the diagonal embedding of $M$ into $(\Gamma(M,\mathcal K\otimes\ell^2S)'')^{\mathcal U}$. 
Moreover, if $M$ and $\mathcal H$ are separable, then we can take $I=S=\mathbb N$.
\end{main}

Theorem \ref{bimod} is of independent interest. In addition to Corollary \ref{hyper}, Theorem \ref{bimod} and its proof lead to characterizations of weak containment in the trivial or coarse bimodule and of relative amenability for inclusions, in the sense of Ozawa and Popa \cite{OP07}. 

\begin{mcor}\label{relamen}
    Let $(M,\tau)$ be a separable tracial von Neumann algebra, $\mathcal H$  a separable $M$-bimodule, $P,Q\subset M$ von Neumann subalgebras,  and $\mathcal U$  a free ultrafilter on $\mathbb N$. Denote $\widehat{M}=M\overline{\otimes}\emph{L}(\mathbb F_\infty)$,  $\overline{M}=M*\emph{L}(\mathbb F_\infty)$, and $\widetilde M=M*_Q(Q\overline{\otimes}\emph{L}(\mathbb Z))$.
    Then the following hold:
    \begin{enumerate}
       
        \item 
     $\mathcal H$ is weakly contained in $\emph{L}^2(M)$ if and only if $\mathcal H\subset\emph{L}^2(\widehat{M}^{\mathcal U})$.
        \item  $\mathcal H$ is weakly contained in $\emph{L}^2(M)\otimes \emph{L}^2(M)$ if and only if $\mathcal H\subset\emph{L}^2(\overline{M}^{\mathcal U})\ominus\emph{L}^2(M^{\mathcal U})$.
         \item  $P$ is amenable relative to $Q$ inside $M$ if and only if there exists $x\in P'\cap \widetilde M^{\mathcal U}$ such that $\emph{E}_{M^{\mathcal U}}(x)=0$ and $\emph{E}_{M^{\mathcal U}}(x^*x)=1$.  
         \item The $M$-bimodule $\emph{L}^2(M)\otimes_Q\emph{L}^2(M)$ is weakly contained in $\emph{L}^2(M)$ if and only if there exists a  $*$-homomorphism $\pi:\widetilde M\rightarrow {\widehat{M}}^{\mathcal U}$ whose restriction to $M$ is the diagonal embedding of $M$. 
    \end{enumerate}
\end{mcor}
\begin{remark}
Using a method from \cite{Ha84,IV14}, Marrakchi  showed that if an $M$-bimodule $\mathcal H$ is weakly contained in $\text{L}^2(M)$, then $\mathcal H\subset\text{L}^2((M\overline{\otimes}\text{L}(\mathbb Z))^{\mathcal U})$, see \cite[Theorem 2.1]{Ma23}.
This improves part (1) of Corollary \ref{relamen}.
We also note that by \cite[Theorem 3.3]{Ma23} the $M$-bimodule $\text{L}^2(M)\otimes_Q\text{L}^2(M)$ is weakly contained in $\text{L}^2(M)$ (i.e., $Q$ is {\it weakly bicentralized} in $M$, in the sense of \cite{BMO19,IM19}) if and only if $(Q'\cap M^{\mathcal U})'\cap M=Q$. 
\end{remark}

Although existential closedness is a strong property, it was noted in \cite{GHS12,FGHS13} that the class of existentially closed separable II$_1$ factors is embedding universal in the following sense:

\begin{main}[\cite{GHS12,FGHS13}]\label{construction}
Any separable II$_1$ factor $M_0$ is contained in an existentially closed separable II$_1$ factor $M$.
\end{main}

Theorem \ref{construction} was noted in \cite{FGHS13}
as a consequence of results from \cite{Us08}. Its proof consists of verifying the model-theoretic definition of existential closedness given in  \cite[Definition 1.1]{FGHS13}. 
In Section \ref{proofs}, we give an operator algebraic presentation of the proof of Theorem \ref{construction}, in which we instead directly verify Definition \ref{ecdef}. 
We hope that this 
will both make the class of existentially closed II$_1$ factors more accessible to the operator algebra community and allow to prove new results. 
In fact, by adapting 
our operator algebraic presentation of the proof of Theorem \ref{construction},
we obtain the following new result:

\begin{main}\label{non-Gamma}
Any separable II$_1$ factor $M_0$ is contained in a non-Gamma separable II$_1$ factor $M$ which satisfies the following: $M$ is existentially closed in any II$_1$ factor $N\supset M$ such that the $M$-bimodule $\emph{L}^2(N)\ominus\emph{L}^2(M)$ is weakly contained in $\emph{L}^2(M)\otimes\emph{L}^2(M)$.
\end{main}

  By \cite{GHS12}, all existentially closed II$_1$ factors are McDuff and hence have property Gamma. Consequently,  a non-Gamma  II$_1$ factor $M$ cannot be  existentially closed in all larger  II$_1$ factors. Moreover, as we will see below, $M$ cannot even be existentially closed in all larger non-Gamma II$_1$ factors.
  
  These facts suggest the following question: how ``close" can a non-Gamma separable II$_1$ factor $M$ be to being existentially closed? 
To make this question more precise, denote by  $\mathcal E_M$ the family of separable II$_1$ factors $N\supset M$ in which  $M$ is  existentially closed.
  We would like to find $M$ for which $\mathcal E_M$ is as large as possible.
If $N\in\mathcal E_M$, then since $M\subset N\subset M^{\mathcal U}$, we get that $M'\cap N^{\mathcal U}=\mathbb C1$. Hence, $M'\cap N=\mathbb C1$ and $N$ is non-Gamma. 
As a consequence,  $\mathcal E_M$ is a proper subset of the set of all non-Gamma separable II$_1$ factors which contain $M$. 
In particular,  this shows that there is no non-Gamma II$_1$ factor which is existentially closed in all non-Gamma II$_1$ factors containing it.
Indeed, the separable II$_1$ factor $P=(M\overline{\otimes}\text{L}(\mathbb Z))*\text{L}(\mathbb Z)$ is non-Gamma and contains $M$, but does not belong to $\mathcal E_M$ since $M'\cap P\not=\mathbb C1$.
More generally, if $N\in\mathcal E_M$, then any II$_1$ subfactor $Q\subset M$ with $Q'\cap M^{\mathcal U}=\mathbb C1$, must satisfy that $Q'\cap N^{\mathcal U}=\mathbb C1$. This altogether shows that for any non-Gamma II$_1$ factor $M$,  inclusions of the form $M\subset N$, with  $N\in\mathcal E_M$, are significantly constrained. 

Nevertheless, Theorem \ref{non-Gamma} shows that there exist non-Gamma II$_1$ factors $M$
for which $\mathcal E_M$ is a large family which contains all II$_1$ factors $N\supset M$ such that $(\star)$ {\it the $M$-bimodule $\emph{L}^2(N)\ominus\emph{L}^2(M)$ is weakly contained in $\emph{L}^2(M)\otimes\emph{L}^2(M)$}.

\begin{remark}\label{optimal}
Condition $(\star)$ in Theorem \ref{non-Gamma} is optimal in the following sense. There are non-Gamma II$_1$ factors $M$ such that $(\star)$ holds for every $N\in\mathcal E_M$. More precisely, 
 Amine Marrakchi (personal communication) proved that if $M=\text{L}(\Gamma)$, where $\Gamma$ is a countable icc group which is biexact in the sense of Ozawa, then the $M$-bimodule  $\text{L}^2(M^{\mathcal U})\ominus\text{L}^2(M)$ is weakly contained in $\text{L}^2(M)\otimes\text{L}^2(M)$.
 In Lemma \ref{freegroup}, we give a direct proof of this fact when $M=\text{L}(\mathbb F_n)$, for $n\geq 2$.
For such a II$_1$ factor $M$, condition ($\star)$ is satisfied by any $N\in\mathcal E_M$ since the $M$-bimodule $\text{L}^2(N)\ominus\text{L}^2(M)$ is contained in $\text{L}^2(M^{\mathcal U})\ominus\text{L}^2(M)$ and thus weakly contained in $\text{L}^2(M)\otimes\text{L}^2(M)$.
\end{remark}

Our last main result provides a spectral gap characterization of Haagerup's approximation property for separable II$_1$ factors, paralleling the characterization of property (T) obtained in \cite{Ta22}.
A separable II$_1$ factor $M$ is said to have {\it Haagerup's (approximation) property} \cite{Ch83,CJ83} if there exists a sequence of completely positive maps $\Phi_n:M\rightarrow M$ such that  $\|\Phi_n(x)-x\|_2\rightarrow 0$, for every $x\in M$, and  $\Phi_n$ is subtracial ($\tau\circ\Phi_n\leq\tau$) and its extension  to $\text{L}^2(M)$ is a compact operator, for every $n$. By  \cite{BF07, OOT15}, Haagerup's property for $M$ is equivalent to the existence of a strictly mixing $M$-bimodule (see Definition \ref{mixingbimod} for this notion) that weakly contains $\text{L}^2(M)$.

\begin{main}\label{hap}
A separable II$_1$ factor $M$ has Haagerup's property if and only there exists a separable II$_1$ factor $\widetilde M$ which contains $M$ such that the $M$-bimodule $\emph{L}^2(\widetilde M)\ominus\emph{L}^2(M)$ is strictly mixing and  $M'\cap \widetilde M^{\mathcal U}\not\subset M^{\mathcal U}$, for a free ultrafilter $\mathcal U$  on $\mathbb N$.
\end{main}


\subsection*{Organization of the paper} Besides the introduction, this paper consists of three other sections.  
Sections \ref{prelim} and \ref{M-valued} are devoted to  von Neumann algebras preliminaries and $M$-valued semicircular systems, respectively. In Section \ref{proofs}, we present the proofs of the results stated in the introduction.

\subsection*{Acknowledgements} We are very grateful to Changying Ding, Daniel Drimbe, Ilijas Farah, Isaac Goldbring, David Jekel, Srivatsav Kunnawalkam Elayavalli, Amine Marrakchi, Jesse Peterson, and Sorin Popa for many comments and conversations. The first named author would also like to thank Stuart White for his kind hospitality at the University of Oxford where this work was completed.

\section{Preliminaries}\label{prelim}
In this section, we record some terminology and basic results concerning tracial von Neumann algebras, bimodules and completely positive maps. We refer  to \cite{AP18} for more information.

\subsection{Tracial von Neumann algebras}
A tracial von Neumann algebra is a pair $(M,\tau)$ consisting of a von Neumann algebra $M$ and a distinguished normal tracial state $\tau:M\rightarrow\mathbb C$. In order to emphasize the dependence of $\tau$ on $M$, we will sometimes write $\tau_M$ instead of $\tau$.
We assume that all inclusions of tracial von Neumann algebras $N\subset M $ are trace preserving, i.e., satisfy $\tau_N={(\tau_M)}_{|N}$.
We let $\|x\|_2=\tau(x^*x)^{1/2}$, for every $x\in M$, and denote by $\text{L}^2(M)$ the Hilbert space obtained by taking the closure of $M$ with respect to $\|\cdot\|_2$. 
For an ultrafilter $\mathcal U$ on a set $I$, we denote by $M^\mathcal U$ the tracial ultraproduct von Neumann algebra endowed with the trace given by $\tau_{\mathcal U}(x)=\lim_{i\rightarrow\mathcal U}\tau(x_i)$, for every $x=(x_i)\in M^{\mathcal U}$.

\subsection{Bimodules} Let $(M,\tau)$ be a tracial von Neumann algebra.
An {\it $M$-bimodule} is a Hilbert space $\mathcal H$ equipped with two normal $*$-homomorphisms $\pi_1:M\rightarrow\mathbb B(\mathcal H)$ and $\pi_2:M^{\text{op}}\rightarrow\mathbb B(\mathcal H)$ whose images commute.
We write $x\xi y=\pi_1(x)\pi_2(y^{\text{op}})\xi$ for $\xi\in\mathcal H$ and $x,y\in M$, and define a $*$-homomorphism $\pi_{\mathcal H}:M\otimes_{\text{alg}} M^{\text{op}}\rightarrow\mathbb B(\mathcal H)$ by letting $\pi_{\mathcal H}(x\otimes y^{\text{op}})=\pi_1(x)\pi_2(y^{\text{op}})$.


Given two $M$-bimodules $\mathcal H,\mathcal K$, we say that $\mathcal H$ is {\it contained} in $\mathcal K$ and  write $\mathcal H\subset\mathcal K$ if there is an $M$-bimodular isometry $T:\mathcal H\rightarrow\mathcal K$.
A basis of neighborhoods for the {\it Fell topology} on the space of all $M$-bimodules is given by $V(\mathcal H,F,S,\varepsilon)$, where $\mathcal H$ is an $M$-bimodule, $F\subset M,S\subset \mathcal H$ are finite sets and $\varepsilon>0$. Here, $V(\mathcal H,F,S,\varepsilon)$ is defined as the set of all $M$-bimodules $\mathcal K$ for which we can find a map $T:S\rightarrow\mathcal K$ such that $$\text{$|\langle xT(\xi)y,T(\eta)\rangle-\langle x\xi y,\eta\rangle|<\varepsilon$, for every $x,y\in F$, and $\xi,\eta\in S$.}$$
If $\mathcal H,\mathcal K$ are $M$-bimodules, we say that $\mathcal H$ is {\it weakly contained} in $\mathcal K$ and write $\mathcal H\subset_{\text{weak}}\mathcal K$ if $\mathcal H$ belongs to the closure of $\mathcal K^{\oplus\infty}:=\mathcal K\otimes\ell^2(\mathbb N)$ in the Fell topology.
Equivalently, $\mathcal H\subset_{\text{weak}}\mathcal K$ if and only if $\|\pi_{\mathcal H}(T)\|\leq\|\pi_{\mathcal K}(T)\|$, for every $T\in M\otimes_{\text{alg}} M^{\text{op}}$.

We next recall \cite[Theorem 3.3.1]{Po86}, which gives a bimodule characterization of property Gamma for separable II$_1$ factors (cf. \cite[Proposition 3.2]{BMO19}). For the reader's convenience, we include a brief proof.
\begin{lemma}[\cite{Po86}]\label{connes}
A separable II$_1$ factor $M$ has property Gamma if and only if $\emph{L}^2(M)\oplus\emph{L}^2(M)$ belongs to the closure of $\emph{L}^2(M)$, in the Fell topology. 
\end{lemma}

\begin{proof} Denote $\xi_1=1\oplus 0,\xi_2=0\oplus 1\in \text{L}^2(M)\oplus\text{L}^2(M)$.

First, assume that $M$ has property Gamma. Let $u_n\in M$ be a sequence of trace zero unitaries such that $\|u_nx-xu_n\|_2\rightarrow 0$, for every $x\in M$. Then $u_n\rightarrow 0$, weakly.
Put $\xi_1^n=1\in\text{L}^2(M)$ and $\xi_2^n=u_n\in\text{L}^2(M)$.  Then we have $\lim_n\langle x\xi_i^ny,\xi_j^n\rangle=\langle x\xi_iy,\xi_j\rangle$, for every $x,y\in M$ and $i,j\in\{1,2\}$.
Thus, $\text{L}^2(M)\oplus\text{L}^2(M)=\overline{\text{span}(M\xi_1M)}\oplus\overline{\text{span}(M\xi_2M)}$ lies in the closure of $\text{L}^2(M)$.

 Conversely, assume that $\text{L}^2(M)\oplus\text{L}^2(M)$ belongs to the closure of $\text{L}^2(M)$ in the Fell topology. Suppose by contradiction that $M$ does not have property Gamma.  Then we can find nets $(\xi_1^n),(\xi_2^n)\subset\text{L}^2(M)$ such that
 $\lim_n\langle x\xi_i^ny,\xi_j^n\rangle=\langle x\xi_iy,\xi_j\rangle$, for every $x,y\in M$ and $i,j\in\{1,2\}$.
Since $x\xi_i=\xi_ix$ and $\|\xi_i\|_2=1$, we get that $\|x\xi_i^n-\xi_i^nx\|_2\rightarrow 0$ and $\|\xi_i^n\|_2\rightarrow 1$, for every $x\in M$ and $i\in\{1,2\}$. Since $M$ does not have property Gamma, \cite[Theorem 2.1]{Co75} implies the existence of a net $(\alpha_i^n)\subset\mathbb T=\{z\in\mathbb C\mid |z|=1\}$ such that $\|\xi_i^n-\alpha_i^n1\|_2\rightarrow 0$, for every $i\in\{1,2\}$.
 Thus,  $|\langle \xi_1^n,\xi_2^n\rangle-\alpha_1^n\overline{\alpha_2^n}|\rightarrow 0$. Since $\langle \xi_1^n,\xi_2^n\rangle\rightarrow\langle\xi_1,\xi_2\rangle=0$ and $(\alpha_1^n\overline{\alpha_2^n})\subset \mathbb T$, this is a contradiction.
\end{proof}

\begin{definition}
A {\it symmetric} $M$-bimodule $(\mathcal H,J)$ is an $M$-bimodule $\mathcal H$ equipped with an anti-unitary involution $J:\mathcal H\rightarrow \mathcal H$ such that $J(x \xi y)=y^*(J\xi)x^*$, for every $x,y\in M$ and $\xi\in\mathcal H$.
\end{definition}

\begin{remark}\label{symmetrize} Let $\mathcal H$ be an $M$-bimodule. Denote by $\overline{\mathcal H}$ the conjugate Hilbert space  endowed with the $M$-bimodule structure  $x\overline{\xi} y=\overline{y^* \xi x^*}$ (see \cite[1.3.7]{Po86}). Then $\mathcal K:=\mathcal H\oplus\overline{\mathcal H}$ is a symmetric Hilbert $M$-bimodule, as witnessed by the anti-unitary involution $J:\mathcal K\rightarrow\mathcal K$ given by $J(\xi\oplus\overline{\eta})=\eta\oplus\overline{\xi}$.\end{remark}

Let $\mathcal H$ be an $M$-bimodule.
A vector $\xi\in\mathcal H$ is called {\it tracial}  if $\langle x\xi,\xi\rangle=\langle\xi x,\xi\rangle=\tau(x)$, for every $x\in M$, and {\it subtracial} if $\langle x\xi,\xi\rangle,\langle \xi x,\xi\rangle\leq\tau(x)$, for every $x\in M$.
A vector $\xi\in\mathcal H$ is called {\it left} (respectively, {\it right}) {\it bounded} if there $C>0$ such that $\|\xi x\|\leq C\|x\|_2$ (respectively, $\|x\xi\|\leq C\|x\|_2$) for every $x\in M$. A vector is {\it bounded} if it is left and right bounded. Note that the subspace of bounded vectors  is dense in $\mathcal H$. If $\xi\in\mathcal H$ is left bounded, we denote by $T_\xi:\text{L}^2(M)\rightarrow\mathcal H$ the bounded operator $T_\xi(x)=\xi x$. 
If $\xi,\eta\in \mathcal H$ are left bounded vectors, then $\langle\xi,\eta\rangle_M:=T_\xi^*T_\eta\in\mathbb B(\text{L}^2(M))$ belongs to $M$. Note that $$\text{$\langle\eta x,\xi\rangle=\tau(\langle \xi,\eta\rangle_Mx)$, for every $x\in M$. }$$

For $M$-bimodules $\mathcal H,\mathcal K$, we denote by $\mathcal H\otimes_M\mathcal K$ their {\it Connes tensor product} (see \cite{Co94}, V, Appendix B).
Let $\mathcal H^0\subset\mathcal H$ be the subspace of left bounded vectors. Then $\mathcal H\otimes_M\mathcal K$ is obtained by separation and completion from the algebraic tensor product $\mathcal H^0\otimes\mathcal K$ endowed with the sesquilinear form $$\langle\xi_1\otimes\eta_1,\xi_2\otimes\eta_2\rangle=\langle\eta_1,\langle\xi_1,\xi_2\rangle_M\eta_2\rangle.$$

\begin{remark}\label{density}
Let  $\mathcal H,\mathcal K$ be $M$-bimodules, $\mathcal H^0\subset\mathcal H$ the subspace of left bounded vectors and $\mathcal K^0\subset\mathcal K$ the subspace of right bounded vectors. 
Let $\xi\in\mathcal H^0$ and $(\eta_i)\subset\mathcal K$ be a net with $\|\eta_i\|\rightarrow 0$. Then $\|\xi\otimes\eta_i\|=\|\eta_i,\langle\xi,\xi\rangle_M\eta_i\|^{1/2}\leq \|\langle\xi,\xi\rangle_M\|^{1/2}\|\eta_i\|\rightarrow 0.$
Similarly, if $(\xi_i)\subset\mathcal H^0$ is a net with $\|\xi_i\|\rightarrow 0$ and $\eta\in\mathcal K^0$, then $\|\xi_i\otimes\eta\|\rightarrow 0$.
These observations imply that if $\mathcal H^1\subset\mathcal H^0$ and $\mathcal K^1\subset\mathcal K^0$ are dense subspaces, then
 the span of $\{\xi\otimes\eta\mid\xi\in\mathcal H^1,\eta\in\mathcal K^1\}$ is dense in $\mathcal H\otimes_M\mathcal K$.
    
\end{remark}

  The following lemma is well-known, and its proof follows a standard recipe (see for instance the proof of \cite[Lemma 13.3.11]{AP18}), but we include details here for completeness.

 \begin{lemma}
 \label{subtracial}
Let $(M,\tau)$ be a tracial von Neumann algebra, $\mathcal H,\mathcal K$ be Hilbert bimodules, $\xi\in\mathcal H$ a subtracial vector and $(\xi_n)\subset\mathcal K$ a net with $\lim_n\langle x\xi_ny,\xi_n\rangle=\langle x\xi y,\xi\rangle$, for every $x,y\in M$. Then there is a net of subtracial vectors $(\eta_n)\in\mathcal K^{\oplus\infty}$  such that $\lim_n\langle x\eta_ny,\eta_n\rangle=\langle x\xi y,\xi\rangle$, for every $x,y\in M$ 
 \end{lemma}

 \begin{proof}
We  follow the proof of \cite[Theorem]{IV14}.
Define normal positive linear functionals $\omega,\omega',\omega_n,\omega_n'$ on $M$ by letting
$\omega(x)=\langle x\xi,\xi\rangle$, $\omega'(x)=\langle \xi x,\xi\rangle$, $\omega_n(x)=\langle x\xi_n,\xi_n\rangle$ and $\omega_n'(x)=\langle \xi_n x,\xi_n\rangle$, for every $x\in M$ and every $n$. Then  $\omega_n\rightarrow\omega$ and $\omega_n'\rightarrow\omega'$, in the 
weak topology on $M_*$. Since the weak and norm closures of convex subsets of $M_*$ coincide,
after replacing $(\omega_n,\omega_n')$  by a convex combination of $(\omega_n,\omega_n')$ and the vectors $(\xi_n)$  by vectors in $\mathcal K^{\oplus\infty}$, we can ensure that in addition to $\lim_n\langle x\xi_ny,\xi_n\rangle=\langle x\xi y,\xi\rangle$, for every $x,y\in M$, we have $\|\omega_n-\omega\|\rightarrow 0$ and $\|\omega_n'-\omega'\|\rightarrow 0$. 

Write $\omega=\tau(\cdot T), \omega'=\tau(\cdot T'),\omega_n=\tau(\cdot T_n)$ and $\omega_n'=\tau(\cdot T_n')$, where $T,T',T_n,T_n'\in\text{L}^1(M)$ are positive elements. Then  $0\leq T,T'\leq 1$ as $\xi$ is subtracial  and $\varepsilon_n:=\|T_n-T\|_1+\|T_n'-T'\|_1\rightarrow 0$.
Let $p_n$ and $p_n'$ be the spectral projections  of $T_n$ and $T_n'$ corresponding to the interval $[0,1+\sqrt{\varepsilon_n}]$.

We claim that there is $n_0$ such that \begin{equation}\label{claim}\text{$\|(1-p_n)\xi_n\|\leq 2\sqrt[4]{\varepsilon_n}$ and $\|\xi_n(1-p_n')\|\leq 2\sqrt[4]{\varepsilon_n}$, for every $n\geq n_0$.}\end{equation}
If  $q_n=1-p_n$, then $q_nT_nq_n\geq (1+\sqrt{\varepsilon_n})q_n$ and $q_nTq_n\leq q_n$, hence $q_n(T_n-T)q_n\geq \sqrt{\varepsilon_n}q_n$. This gives $\sqrt{\varepsilon_n}\|q_n\|_1\leq \|T_n-T\|_1\leq \varepsilon_n$, hence $\|q_n\|_1\leq\sqrt{\varepsilon_n}$ and $\|q_nT_n\|_1\leq \|T_n-T\|_1+\|q_n\|_1\leq \varepsilon_n+\sqrt{\varepsilon_n}$. Thus, $\|(1-p_n)T_n\|_1=\|q_nT_n\|_1\leq 2\sqrt{\varepsilon_n}$ and therefore $\|(1-p_n)\xi_n\|=\sqrt{\|(1-p_n)T_n\|_1}\leq 2\sqrt[4]{\varepsilon_n}$, for every $n$ large enough. This and a similar inequality prove the claim. 

Let $c_n=(1+\sqrt{\varepsilon_n})^{-1/2}$ and define $\eta_n=c_np_n\xi_np_n'$.
Since $p_nT_np_n\leq (1+\sqrt{\varepsilon_n})$, for every $x\in M$, we have that $\|x\eta_n\|\leq c_n\|xp_n\xi_n\|=c_n\sqrt{\tau(x^*xp_nTp_n)}\leq\sqrt{\tau(x^*x)}=\|x\|_2$ and similarly that $\|\eta_nx\|\leq \|x\|_2$.
Thus, $\eta_n$ is subtracial. Since $\|\xi_n-\eta_n\|\leq (1-c_n)+\|\xi_n-p_n\xi_np_n'\|$ and $\lim_n c_n=1$, \eqref{claim} implies that $\lim_n\|\xi_n-\eta_n\|=0$. Thus, $\lim_n\langle x\eta_ny,\eta_n\rangle=\lim_n \langle x\xi_ny,\xi_n\rangle=\langle x\xi y,\xi\rangle$, for every $x,y\in M$, which finishes the proof of the lemma.
 \end{proof}

\begin{proposition}\label{classA} 
Let $\mathcal A$ be the class of separable II$_1$ factors $M$ for which  $\emph{L}^2(M)$ weakly contains every $M$-bimodule.
Let $M$ and $M_n$, $n\in\mathbb N$, be separable II$_1$ factors. Then the following hold:
    \begin{enumerate}
        \item If $M$ is existentially closed, then $M \in \mathcal A$.
        \item The hyperfinite II$_1$ factor $R \in \mathcal A$.
        \item If $M \in \mathcal A$, then $M^t \in \mathcal A$ for every $t > 0$.
        \item If $M_n \subset M_{n+1}$ and $M_n \in \mathcal A$, for every $n\in\mathbb N$, then $\overline{\cup_{n\in\mathbb N} M_n}^{\mathrm{WOT}} \in \mathcal A$.
        \item If $M \in \mathcal A$, then $M\overline{\otimes} R \in \mathcal A$.
    \end{enumerate}
\end{proposition}
\begin{proof}
    (1) and (2) follow from Theorem \ref{ec} and \cite[Proposition 2.3.2.]{Po86} (see also Corollary \ref{hyper}), respectively.

    (3) Assume that $M\in\mathcal A$ and denote $N = M^t$, for some $t>0$. Let $\mathcal H$ be an $N$-bimodule, $d \ge 1/t$ be a integer, and $p$ be a projection in $\mathbb M_{d} (\mathbb C) \overline{\otimes} N$ with $(\Tr \otimes \tau) (p) = 1/t$, where $\Tr$ is the standard (non-normalized) trace on $\mathbb M_{d} (\mathbb C)$ and $\tau$ denote the trace of $N$. Then $M \simeq N^{1/t} \simeq p(\mathbb M_{d} (\mathbb C) \overline{\otimes} N)p$. Since the isomorphism class of $N^{1/t}$ only depends on  $(\Tr \otimes \tau)(p)$, we can take $p = \mathrm{diag} (p_1,\cdots,p_d)$ where $p_1,\cdots,p_d$ are projections in $N$. Given an $N$-bimodule $\mathcal H$, we consider the $M$-bimodule $\mathcal K = p ( \mathbb M_{d} (\mathbb C) \otimes \mathcal H) p$ with the natural left and right $p(\mathbb M_{d} (\mathbb C) \overline{\otimes} N)p$-actions. More specifically, $$x \cdot \xi \cdot y = 
    \big((\sum\limits_{ 1\le k, l \le n} x_{i,k} \cdot \xi_{k,l} \cdot  x_{l,j} )_{i,j}\big),$$ for all $x=(x_{i,j})$, $y=(y_{i,j}) \in p(\mathbb M_{d} (\mathbb C) \overline{\otimes} N)p$,  $\xi=(\xi_{i,j}) \in \mathcal K$, where $x_{i,j},y_{i,j}\in p_iNp_j$, $\xi_{i,j}\in p_i\mathcal Hp_j$.
    
To show that $\mathcal H$ is weakly contained in $\text{L}^2(N)$, it suffices to find, given $F\subset N, S\subset\mathcal H$ finite and $\varepsilon >0$,  a map $T:S\rightarrow  \text{L}^2(N)^{\oplus\infty}$, such that $| \langle x \xi x' , \eta  \rangle -\langle  x T(\xi) x' , T(\eta) \rangle | < \varepsilon$, for every $x,x'\in F$ and $\xi,\eta\in S$. 
    If $(\Tr \otimes \tau)(p) \ge 1$, we can assume that $p_1 =1$, and the conclusion follows from applying the fact that $\mathcal K$ is weakly contained in $\text{L}^2(M)$ as an $M$-bimodule
    to $e_{1,1} \otimes x, e_{1,1} \otimes x' \in  M=p ( \mathbb M_{d} (\mathbb C) \overline{\otimes} N ) p$, and $e_{1,1} \otimes \xi, e_{1,1} \otimes \eta \in \mathcal K=p ( \mathbb M_{d} (\mathbb C) \otimes \mathcal H ) p$, for $x,y\in F$ and $\xi,\eta\in S$.
    
    If $(\Tr \otimes \tau)(p) < 1$, then we can assume that 
    $p \in N$ and $M = pNp$. Let $\sum q_i = 1$ be a finite partition of $1$ by projections in $N$ such that $\tau (q_i) \le \tau (p) $. For each $i$, let $v_i$ be a partial isometry in $N$ such that $v_i v_i^* = q_i$ and $v_i^* v_i \le p$. Let $\xi,\eta\in S$.
    Then we have \begin{align*}
        \langle x \xi x' , \eta  \rangle &= \sum\limits_{k,i,j,l} \langle (q_k x q_i)   (q_i \xi  q_j)   (q_j x' q_l) ,  q_k \eta q_l \rangle= \sum\limits_{k,i,j,l} \langle (v^*_k x  v_i)   (v^*_i  \xi   v_j)   (v^*_j x'  v_l) ,  v^*_k  \eta  v_l \rangle.\end{align*}

  Since $p\mathcal H p$ is weakly contained in   $\text{L}^2(M)=\text{L}^2(pNp)$ as an $M$-bimodule, for some $\varepsilon_{i,j,k,l}>0$ with $\sum\limits_{i,j,k,l} \varepsilon_{i,j,k,l} < \varepsilon$, there exists $T_{i,j}(\xi) \in \text{L}^2(pNp)^{\oplus\infty}$ for every $(i,j)$ and $\xi \in S$, such that for every $(i,j,k,l)$ we have
    $$| \langle (v^*_k x  v_i)   (v^*_i  \xi   v_j)   (v^*_j x'  v_l) ,  v^*_k  \eta  v_l \rangle- \langle (v^*_k  x  v_i)   T_{i,j}(\xi)   (v^*_j  x' v_l) ,  T_{k,l}(\eta) \rangle|< \varepsilon_{i,j,k,l}, \forall x,x'\in F,\xi,\eta\in S.$$ 
Note that   
    $\langle (v^*_k   x  v_i  )   T_{i,j}(\xi)   (  v^*_j   x'   v_l  ) ,  T_{k,l}(\eta) \rangle = 
    \langle   x  (v_i    T_{i,j}(\xi)     v^*_j )  x'    ,  v _k T_{k,l}(\eta)v^*_l \rangle$. Let
 $T(\xi) = \sum\limits_{i,j} v_i T_{i,j}(\xi) v^*_j \in \text{L}^2(N)^{\oplus\infty}$ for  $\xi\in S$. Then we have
    \begin{align*}
        &|\langle x \xi x' , \eta  \rangle -\langle  x T(\xi) x' , T(\eta) \rangle |  \\
        \le& \sum\limits_{i,j,k,l}  | \langle (v^*_k x  v_i)   (v^*_i  \xi   v_j)   (v^*_j x'  v_l) ,  v^*_k  \eta  v_l \rangle - \langle  x (v_i T_{i,j}(\xi) v^*_j)  x' , v_k T_{k,l}(\eta)) v^*_l \rangle  |< \varepsilon
    \end{align*}
for every $\xi,\eta\in S$, which finishes the proof of (3).

(4) Let $M = \overline{\cup M_n}^{\rm WOT}$ and $\text{E}_n$ be the conditional expectation from $M$ onto $M_n$, for every $n\in\mathbb N$. Let $\mathcal H$ be an $M$-bimodule. In order to show that $\mathcal H$ is weakly contained in $\text{L}^2(M)$,  it suffices to prove that for every $F\subset M, S\subset \mathcal H$ finite and $\varepsilon>0$ we can find a map $T:S\rightarrow\text{L}^2(M)^{\oplus\infty}$ such that $| \langle x \xi y , \xi  \rangle -\langle  x T(\xi) y , T(\xi) \rangle | < \varepsilon$, for every $x,y\in F$ and $\xi\in S$. To this end, we may moreover assume that $\|x\|\leq 1$, for every $x\in F$, and that $S$ consists of subtracial vectors.

Let $n\in\mathbb N$ such that $\|\text{E}_n(x)-x\|_2<\varepsilon/6$, for every $x\in F$. Then 
$|\langle x\xi y,\xi\rangle-\langle \text{E}_n(x)\xi\text{E}_n(y),\xi\rangle|<\varepsilon/3$, for every $x,y\in F$ and $\xi\in S$.
Since $M_n\in\mathcal A$, $\mathcal H$ is weakly contained in $\text{L}^2(M_n)$ as an $M_n$-bimodule. If $\xi\in S$, then since $\xi$ is a subtracial vector, Lemma \ref{subtracial} gives a subtracial vector $T(\xi)\in \text{L}^2(M_n)^{\oplus\infty}$ such that $|\langle \text{E}_n(x)\xi\text{E}_n(y),\xi\rangle-\langle\text{E}_n(x)T(\xi)\text{E}_n(y),T(\xi)\rangle|<\varepsilon/3$, for every $x,y\in F$. If we view $\text{L}^2(M)^{\oplus\infty}$ as an $M$-bimodule, then $T(\xi)\in\text{L}^2(M_n)^{\oplus\infty}\subset\text{L}^2(M)^{\oplus\infty}$ is still a subtracial vector. Since $\|\text{E}_n(x)-x\|_2<\varepsilon/6$, for every $x\in F$, we get that
$|\langle\text{E}_n(x)T(\xi)\text{E}_n(y),T(\xi)\rangle-\langle xT(\xi)y,T(\xi)\rangle|<\varepsilon/3$, for every $x,y\in F$. Altogether, we get that $|\langle x\xi y,\eta\rangle-\langle xT(\xi)y,T(\eta)\rangle|<\varepsilon$, for every $x,y\in F$ and $\xi\in S$, which finishes the proof of (4).

(5) Write $R=\overline{\cup_{n\in\mathbb N}\mathbb M_{2^n}(\mathbb C)}^{\rm WOT}$, for the usual inclusions $\mathbb M_{2^n}(\mathbb C)\subset \mathbb M_{2^{n+1}}(\mathbb C)$, for every $n\in\mathbb N$.  Then $M \overline{\otimes} R  =  \overline{\cup_{n\in\mathbb N}(M \overline{\otimes} \mathbb M_{2^n}(\mathbb C))}^{\rm WOT}$, and the conclusion follows from (3) and (4).
\end{proof}

\subsection{Completely positive maps}
Let $(M,\tau)$ be a tracial von Neumann algebra. 
We now recall the well-known correspondence between $M$-bimodules $\mathcal H$ and normal, completely positive maps on $M$.
Let $\xi\in\mathcal H$ be a bounded vector. Then $\Phi_\xi(x)=T_\xi^*xT_\xi$ belongs to $M$, for every $x\in M$. The map $\Phi_\xi:M\rightarrow M$, called a {\it coefficient} {of $\mathcal H$}, is normal, completely positive and satisfies $$\text{$\tau(\Phi_\xi(x)y)=\langle x\xi y,\xi\rangle$, for every $x,y\in M$.}$$
Moreover, $\Phi_\xi$ extends to a bounded operator $\Phi_\xi:\text{L}^2(M)\rightarrow\text{L}^2(M)$. Let $C>0$ such that $\|x\xi\|,\|\xi x\|\leq C\|x\|_2$, for every $x\in M$. Then  $|\tau(\Phi_\xi(x)y)|=|\langle x\xi,\xi y^*\rangle|\leq \|x\xi\|\;\|\xi y^*\|\leq C^2\|x\|_2\;\|y\|_2$, for every $x,y\in M$. Thus, we get that $\|\Phi_\xi(x)\|_2\leq C^2\|x\|_2$, for every $x\in M$. 

 \begin{definition}\label{mixingbimod}
Let $\mathcal H$ be an $M$-bimodule.

\begin{enumerate}
\item We call $\mathcal H$ a {\it mixing} $M$-bimodule if for every sequence $u_n\in\mathcal U(M)$ with $u_n\rightarrow 0$ weakly we have $\lim\limits_{n\rightarrow\infty}\big(\sup_{x\in M,\|x\|\leq 1}|\langle u_n\xi x,\eta\rangle|\big)=\lim\limits_{n\rightarrow\infty}\big(\sup_{x\in M,\|x\|\leq 1}|\langle x\xi u_n,\eta\rangle|\big)=0$, for every $\xi,\eta\in\mathcal H$.
\item Denote by $\mathcal H_{\text{mix}}$ the set of bounded vectors $\xi\in\mathcal H$ such that $\Phi_\xi\in\mathbb B(\text{L}^2(M))$ is  compact. We call $\mathcal H$  a {\it strictly mixing} $M$-bimodule if the span of $M\mathcal H_{\text{mix}}M$ is dense in $\mathcal H$. 
\end{enumerate}
\end{definition}
The notions of mixing and strictly mixing bimodules were introduced respectively in \cite[Definition 2.3]{PS09} and \cite[Definition 4]{OOT15} (see also \cite[Definition 3.1]{BF07}). These two notions have recently been shown to be equivalent in \cite[Theorem 5.10]{DKEP22}.

For future reference we note that if $\xi\in\mathcal H$ is subtracial, then $\Phi_\xi$ is subunital and subtracial.


Conversely, given a normal, completely positive map $\Phi:M\rightarrow M$, there is a Hilbert $M$-bimodule $\mathcal H_{\Phi}$ together with a vector $\xi_{\Phi}\in\mathcal H_{\Phi}$ 
such that $\overline{\text{span}(M\xi_{\Phi}M)}=\mathcal H_{\Phi}$ and $\langle x\xi_{\Phi}y,\xi_{\Phi}\rangle=\tau(\Phi(x)y)$, for every $x,y\in M$. 
Assume that $\Phi$ is symmetric, i.e.,  $\tau(\Phi(x)y)=\tau(x\Phi(y))$, for every $x,y\in M$. Then $\mathcal H_{\Phi}$ is a symmetric $M$-bimodule  as witnessed by the anti-unitary involution  $J:\mathcal H_{\Phi}\rightarrow \mathcal H_{\Phi}$ given by $J(x\xi_{\Phi}y)=y^*\xi_{\Phi}x^*$.
Moreover, $\xi_{\Phi}$ is a bounded vector, $J(\xi_{\Phi})=\xi_{\Phi}$, and 
\begin{equation}\label{Phi}
    \text{$\langle \xi_{\Phi},x\xi_{\Phi}y\rangle_M=\Phi(x)y$, for every $x,y\in M$.}
\end{equation}

\begin{lemma}\label{trivialbimod}
Let $(M,\tau)$ be a tracial von Neumann algebra, $\Phi:M\rightarrow M$ a normal completely positive map, and $\mathcal K$ an $M$-bimodule.
Then $\mathcal H_{\Phi}\subset_{\emph{weak}}\mathcal K$ if and only if there is a net $(\Phi_i)$ of coefficients of $\mathcal K^{\oplus\infty}$ such that $\|\Phi_i(x)-\Phi(x)\|_2\rightarrow 0$, for every $x\in M$. 

Moreover, if $\Phi$ is subunital, subtracial and $\mathcal H_{\Phi}\subset_{\emph{weak}}\mathcal K$, then there is a net $(\Phi_i)$ of subunital, subtracial coefficients of $\mathcal K^{\oplus\infty}$ such that $\|\Phi_i(x)-\Phi(x)\|_2\rightarrow 0$, for every $x\in M$. Furthermore, in this case, if $\Phi$ is symmetric, then we can assume that $\Phi_i$ is symmetric, for every $i$.

Finally, assume that $M$ is separable. Then we may take $(\Phi_i)$ to be a sequence in the above assertions.
\end{lemma}

\begin{proof}
First, we prove the ``if" part of the main assertion. Assume that $(\Phi_i)$  is a net of coefficients of $\mathcal K^{\oplus\infty}$ so that $\|\Phi_i(x)-\Phi(x)\|_2\rightarrow 0$, for every $x\in M$. 
Let $\xi_i\in\mathcal K^{\oplus\infty}$ such that $\tau(\Phi_i(x)y)=\langle x\xi_iy,\xi_i\rangle$.
 Since $\lim_i\tau(\Phi_i(x)y)=\tau(\Phi(x)y)$, we conclude that $\langle x\xi_\Phi y,\xi_\Phi\rangle=\lim_i\langle x\xi_iy,\xi_i\rangle$, for every $x,y\in M$. This implies that $\mathcal H_{\Phi}\subset_{\text{weak}}\mathcal K$.

Now, assume that $\Phi$ is subunital, subtracial and $\mathcal H_{\Phi}\subset_\text{weak}\mathcal K$. 
Then $\xi_\Phi\in\mathcal H_{\Phi}$ is a subtracial vector. By Lemma \ref{subtracial} we  find a net of subtracial vectors $(\eta_i)\subset\mathcal K^{\oplus\infty}$ with $\langle x\xi_\Phi y,\xi_\Phi\rangle=\lim_i\langle x\eta_i y,\eta_i\rangle$, for every $x,y\in M$. Thus, if $\Phi_i:=\Phi_{\eta_i}$, then $\tau(\Phi(x)y)=\lim_i\tau(\Phi_i(x)y)$, for every $x,y\in M$. 
Since $\Phi_i$ is completely positive, $\Phi_i(x)^*\Phi_i(x)\leq \Phi_i(x^*x)$ and thus $\|\Phi_i(x)\|\leq \|x\|$, for every $x\in M$ and every $i$.
This implies that $\Phi_i(x)\rightarrow\Phi(x)$, in the weak topology on $\text{L}^2(M)$, for every $x\in M$. Since the set of subunital, subtracial coefficients of $\mathcal K^{\oplus\infty}$   is a convex set, after taking convex combinations of $(\Phi_i)$, the moreover assertion follows. Assume additionally that $\Phi$ is symmetric.
Since $\eta_i$ is subtracial,  there is a completely positive map $\Phi_i^*:M\rightarrow M$ such that $\langle x\eta_i y,\eta_i\rangle=\tau(x\Phi_i^*(y))$, for every $x\in M$ and every $i$. Let $\Psi_i=\frac{1}{2}\big(\Phi_i+\Phi_i^*):M\rightarrow M$. Then $\Psi_i(x)\rightarrow\Phi(x)$, in the weak topology on $\text{L}^2(M)$, for every $x\in M$. Since $\Psi_i$ is symmetric for every $i$ and the set of symmetric, subunital, subtracial coefficients of $\mathcal K^{\oplus\infty}$   is a convex set, the furthermore assertion follows by taking convex combinations of $(\Psi_i)$.

Next, we prove the ``only if" part of the main assertion. let $\Phi:M\rightarrow M$ be a normal completely positive map such that $\mathcal H_{\Phi}\subset_{\text{weak}}\mathcal K$. 
Let $T\in\text{L}^1(M)$ be a positive element such that $\tau\circ\Phi=\tau(\cdot T)$. 
For $n\in\mathbb N$, let $T_n\in M$ such that $\|T_n-T\|_1\leq (8n^2)^{-1}$. 
Let $p_n\in M$ be a projection such that $p_nTp_n\in M$ and 
$\|p_n-1\|_2\leq (8n^2(\|T_n\|+1))^{-1/2}$.
 Define $\Psi_n:M\rightarrow M$ by  
 $\Psi_n(x)=\Phi(p_n xp_n)$.
Let $x\in (M)_1$. Since  $\|\Phi(x)\|_2^2=\tau(\Phi(x)^*\Phi(x))\leq \tau(\Phi(x^*x))=\tau(x^*xT)$, $\|p_nxp_n-x\|\leq 2$, and $\|p_nxp_n-x\|_2\leq 2\|p_n-1\|_2$, we get that

\begin{align*}
\|\Psi_n(x)-\Phi(x)\|_2^2\leq\tau((p_nxp_n-x)^*(p_nxp_n-x)T)\leq 4\|T_n-T\|_1+4\|T_n\|\|p_n-1\|_2^2\leq (n^2)^{-1}.
\end{align*}

Let $c_n:=\max\{\|\Psi_n(1)\|,\|p_n Tp_n\|\}>0$. Since $\tau\circ\Psi_n=\tau(\cdot (p_nTp_n))$, we get that $c_n^{-1}\Psi_n$ is a subunital, subtracial completely positive map.
 By applying the moreover assertion to $c_n^{-1}\Psi_n$ and using that  $\sup_{x\in (M)_1}\|\Psi_n(x)-\Phi(x)\|_2\leq 1/n$, we find a net $(\Phi_i)$ of coefficients of $\mathcal K^{\oplus\infty}$ such that $\|\Phi_i(x)-\Phi(x)\|_2\rightarrow 0$, for every $x\in M$.

 Finally, assume that $M$ is separable. Let $\Phi,(\Phi_i)$ be subunital, subtracial completely positive maps on $M$ such that $\|\Phi_i(x)-\Phi(x)\|_2\rightarrow 0$, for every $x\in M$. Since $\|\Phi(x)\|_2\leq \|x\|_2$ and $\|\Phi_i(x)\|_2\leq \|x\|_2$, we get that $\|\Phi_i(x)-\Phi(x)\|_2\leq 2\|x\|_2$, for every $x\in M$. It follows that we can find a subsequence $(\Phi_{i_n})$ of $(\Phi_i)$ such that $\|\Phi_{i_n}(x)-\Phi(x)\|_2\rightarrow 0$, for every $x\in M$. This implies that $(\Phi_i)$ can be taken to be a sequence in the assertions of Lemma \ref{trivialbimod}.
 \end{proof}


\subsection{An elementary lemma on homomorphisms to ultrapowers}
We conclude this section by recording some terminology and  a lemma that will be used in  the proofs of Theorems \ref{construction} and \ref{non-Gamma}.
Let $S$ be a set.
A {\it $*$-monomial} in variables $X_s,s\in S$, is an expression of the form $Y_1Y_2\cdots Y_k$, where $k\in\mathbb N$ and $Y_i\in\{X_s,X_s^*\mid s\in S\}$, for every $1\leq i\leq k$. A {\it $*$-polynomial} $p$ in variables $X_s,s\in S,$ is a complex linear combination of $*$-monomials.
If $A$ is a $*$-algebra, we denote by $p(x_s, s\in S)$ the {\it evaluation} of $p$ at some $\{x_s\mid s\in S\}\subset A$. 

\begin{lemma}\label{hom}
Let  $(M,\tau_M)$ and $(N,\tau_N)$ be tracial von Neumann algebras, $\{x_s\mid s\in S\}\subset (N)_1$  a generating set of $N$, and $\mathcal U$ be a cofinal ultrafilter on a directed set $I$.
For every $i\in I$, let $\{x_{i,s}\mid s\in S\}\subset (M)_1$ such that $\lim_i\tau_M(p(x_{i,s},s\in S))=\tau_N(p(x_s, s\in S))$, for every $*$-monomial $p$ in variables $X_s,s\in S$.
 Then there is a trace preserving $*$-homomorphism $\pi:N\rightarrow M^{\mathcal U}$ such that $\pi(x_s)=(x_{i,s})$, for every $s\in S$.

\end{lemma}

\begin{proof}
Let  $A=\{p(x_s, s\in S)\mid \text{$p$ $*$-polynomial in $X_s,s\in S$}\}$ and define a map $\pi:A\rightarrow M^\mathcal U$ by letting $\pi(p(x_s,s\in S))=(p(x_{i,s},s\in S)),$ for every $*$-polynomial $p$ in $X_s,s\in S$.

Then 
$\pi$ is well-defined. If $p_1(x_s,s\in S) = p_2(x_s,s\in S)$, for $*$-polynomials $p_1,p_2$ on $X_s,s\in S$, then $p(x_s,s\in S)=0$, with $p=p_1-p_2$. Thus, we have that $$\lim_{i\rightarrow\mathcal U}\|p(x_{i,s},s\in S)\|_{2,\tau_M}=\lim\limits_{i\rightarrow\mathcal U}\tau_M(p^*p(x_{i,s},s\in S))^{1/2}=\tau_N(p^*p(x_s,s\in S))^{1/2}=0.$$ 

Hence,  $\lim_{i\rightarrow\mathcal U}\|p_1(x_{i,s},s\in S)-p_2(x_{i,s},s\in S)\|_{2,\tau_M}=0$, so $\pi(p_1(x_s,s\in S))=\pi(p_2(x_s,s\in S))$ in $M^\mathcal U$. Since $\pi$ is a trace preserving $*$-homomorphism on $A$ and $A$ is SOT-dense in $N$, $\pi$ extends to a trace preserving  $*$-homomorphism $\pi:N\rightarrow M^{\mathcal U}$. By definition, $\pi(x_s)=(x_{i,s})$, for every $s\in S$.
\end{proof}


\section{Shlyakhtenko's $M$-valued semicircular systems
}
\label{M-valued}
In \cite{Sh97}, Shlyakhtenko introduced a construction  which generalizes Voiculescu's free Gaussian functor (\cite{Vo83}) and  associates to every  von Neumann algebra $M$ and symmetric $M$-bimodule $\mathcal H$, a von Neumann algebra which contains $M$. We recall this construction here in the case when $M$ is tracial, following closely \cite[Section 3]{KV15}.

\begin{definition}
Let $(M,\tau)$ be a tracial von Neumann algebra, $(\mathcal H,J)$ a symmetric $M$-bimodule and denote by $\mathcal H^0\subset\mathcal H$ the set of bounded vectors. Define the Fock space associated to $\mathcal H$ by $$\mathcal F_M(\mathcal H)=\text{L}^2(M)\oplus\Big(\bigoplus_{n=1}^{\infty}\mathcal H^{\otimes^n_M}\Big).$$ For every $\xi\in\mathcal H^0$, we define $\ell(\xi)\in\mathbb B(\mathcal F_M(\mathcal H))$ by letting $\ell(\xi)(x)=\xi x$ and $$\text{$\ell(\xi)(\xi_1\otimes_M\cdots\otimes_M\xi_n)=\xi\otimes_M\xi_1\otimes_M\cdots\otimes_M\xi_n$, for every $x\in M$ and $\xi_1,\cdots,\xi_n\in\mathcal H^0$.}$$ 
Then $\ell(\xi)^*(x)=0$ and $$\text{$\ell(\xi)^*(\xi_1\otimes_M\cdots\otimes_M\xi_n)=\langle\xi,\xi_1\rangle_M\xi_2\otimes_M\cdots\otimes_M\xi_n$, for every $x\in M$ and $\xi_1,\cdots,\xi_n\in\mathcal H^0$.}$$

For $\xi\in\mathcal H^0$ with $J(\xi)=\xi$, we denote $s(\xi)=\ell(\xi)+\ell(\xi)^*$. Since $\ell(\xi)^*\ell(\xi)=\langle \xi,\xi\rangle_M$, we get that $\|\ell(\xi)\|=\|\langle\xi,\xi\rangle_M\|^{1/2}$ and thus $\|s(\xi)\|\leq 2\|\langle\xi,\xi\rangle_M\|^{1/2}.$

Shlyakhtenko's $M$-valued semicircular system associated to $\mathcal H$ is then defined as
$$\Gamma(M,\mathcal H)'':=\Big(M\cup\{s(\xi)\mid \xi\in \mathcal H^0, J(\xi)=\xi\}\Big)''\subset\mathbb B(\mathcal F_M(\mathcal H)).$$
Let $\Omega\in\mathcal F_M(\mathcal H)$ be the vacuum unit vector given by $\Omega=1\in\text{L}^2(M)$. Then  $\tau:\Gamma(M,\mathcal H)''\rightarrow\mathbb C$ given by $\tau(x)=\langle x\Omega,\Omega\rangle$ is a faithful normal tracial state (see \cite{Sh97} and \cite[Proposition 3.2]{KV15}) The map $\Gamma(M,\mathcal H)''\ni x\mapsto x\Omega\in\mathcal F_M(\mathcal H)$ extends to a unitary operator $\text{L}^2(\Gamma(M,\mathcal H)'')\rightarrow \mathcal F_M(\mathcal H)$.

 \end{definition}


\begin{lemma}\label{SOTapprox}
In the above notation, let  $\cH_1 \subseteq \cH^0$ be a dense subspace such that $M \cH_1 M \subset \cH_1$. Then for every $\eta \in \cH^0$ with $J \eta = \eta$, there exists a net $(\eta_n) \subset \cH_1$ such that $||\eta_n - \eta || \to 0$, $\sup_n ||s(\eta_n)|| < \infty$ and $s(\eta_n) \to s(\eta)$ in the SOT.
\end{lemma}
\begin{proof}
Since $\eta \in \cH^0$, we can assume $||T_{\eta}|| \le 1$. Let $(\xi_n) \subset \cH_1$ such that $||\xi_n - \eta|| \to 0$, and define $s_n  = f(_M\langle \xi_n, \xi_n \rangle),t_n = f(\langle \xi_n, \xi_n \rangle_M)$ respectively, where $f(t) = \begin{cases} 1 & \text{if $0\le t \le1$} \\ t^{-\frac{1}{2}} & \text{if $t \ge 1$} 
\end{cases}$. 

Let $\eta_n = s_n \xi_n t_n \in M\cH_1M\subset\cH_1$. Then $\eta_n$ is subtracial and thus $\|s(\eta_n)\|\leq 2$, for every $n$. We claim that  $||\eta_n -\xi_n||\to 0 $.
To this end, note that $$||\eta_n -\xi_n|| \le ||(1-s_n) \xi_n || + || s_n \xi_n (1-t_n) ||\leq || (1-s_n)\xi_n ||+|| \xi_n(1-t_n) ||.$$ Let  $p_n$ be the spectral projection  of $ _M\langle \xi_n, \xi_n  \rangle$ corresponding to the interval $[1, \infty)$. Since we have that
$1_{[1,\infty)}(t) (t-1) \ge (1-f(t))^2t$, we get that  $(1-s_n)^2 _M\langle \xi_n, \xi_n \rangle\leq p_n(_M\langle\xi_n,\xi_n\rangle-1)$ and therefore
\begin{align*}
    ||(1-s_n) \xi_n || ^2 &= \tau( (1-s_n)^2 _M\langle \xi_n, \xi_n \rangle)
\\
&\le \tau(p_n  {}_M\langle \xi_n, \xi_n \rangle - p_n  )\\
&\le \tau(p_n ( {}_M\langle \xi_n, \xi_n \rangle - {}_M\langle \eta, \eta \rangle  ) ) \\
&\le ||  {}_M\langle \xi_n, \xi_n \rangle - {}_M\langle \eta, \eta \rangle ||_1 \\
&\le ||{}_M\langle \xi_n, \xi_n - \eta \rangle ||_1 + || {}_M\langle \xi_n - \eta, \eta \rangle||_1\\
&\le (||\xi_n -\eta||_2 + 2|| \eta||_2)||\xi_n -\eta||_2,
\end{align*}
 where we used that  $\tau(p_n{}_M\langle\eta,\eta\rangle)\leq\tau(p_n)$ as ${}_M\langle \eta, \eta \rangle\leq 1$. Similarly, $$|| \xi_n (1-t_n)||^2 \le (||\xi_n -\eta||_2 + 2|| \eta||_2)||\xi_n -\eta||_2.$$ Thus,
 $||\eta_n -\xi_n|| \le 2 (||\xi_n -\eta||_2 + 2|| \eta||_2)^{1/2}||\xi_n -\eta||_2^{1/2} $, which implies the claim.
 The claim gives $\|\eta_n-\eta\|\rightarrow 0$. As $\sup_n\|s(\eta_n)\|<\infty$, it follows that $s(\eta_n) \to s(\eta)$ in the SOT.
\end{proof}

\begin{lemma}
In the above notation, let 
$\mathcal H^1\subset\mathcal H^0$ be a dense subspace such that  $J(\mathcal H^1)=\mathcal H^1$. Then 
$\Gamma(M,\mathcal H)''=\Big(M\cup\{s(\xi)\mid \xi\in \mathcal H^1, J(\xi)=\xi\}\Big)''$. 

Moreover, let $\{\xi_i\}_{i\in I}\subset\mathcal H^0$ be a family of vectors such that $J(\xi_i)=\xi_i$, for every $i\in I$, and the span of $\{M\xi_iM\}_{i\in I}$ is dense in $\mathcal H$. Then $\Gamma(M,\mathcal H)''=\Big(M\cup\{s(\xi_i)\mid i\in I\}\Big)''$.
\end{lemma}

\begin{proof} 
The main assertion follows by applying Lemma \ref{SOTapprox} to  $M\cH^1M$. 

We give an alternative proof of the main assertion which does not rely on Lemma \ref{SOTapprox}.
We let $\mathcal K\subset\mathcal F_M(\mathcal H)$ be the span of $\text{L}^2(M)\cup(\bigcup_{n\geq 1}\{\xi_1\otimes_M\cdots\otimes_M\xi_n\mid\xi_1,\cdots,\xi_n\in\mathcal H^1\})$ and denote $\mathcal M=\Big(M\cup\{s(\xi)\mid \xi\in \mathcal H^1, J(\xi)=\xi\}\Big)''.$
We claim that $\mathcal K\subset \mathcal M\Omega$. To this end, for $n\geq 1$, denote by $\mathcal K_n\subset\mathcal F_M(\mathcal H)$ the span of $\text{L}^2(M)\cup(\bigcup_{1\leq k\leq n}\{\xi_1\otimes_M\cdots\otimes_M\xi_k\mid\xi_1,\cdots,\xi_k\in\mathcal H^1\})$. Proceeding by induction, assume that $\mathcal K_n\subset\mathcal M\Omega$, for $n\geq 1$.
If $\xi,\xi_1,\cdots,\xi_n\in\mathcal H^1$ and $J(\xi)=\xi$, then 
$\xi\otimes_M\xi_1\otimes_M\cdots\otimes_M\xi_n=s(\xi)(\xi_1\otimes_M\cdots\otimes_M\xi_n)-\langle\xi,\xi_1\rangle_M(\xi_2\otimes_M\cdots\otimes_M\xi_n)\in s(\xi)\mathcal M\Omega+M(\mathcal M\Omega)=\mathcal M\Omega$. 
Since every $\eta\in\mathcal H^1$ can be written as $\eta=\eta_1+i\eta_2$, where $\eta_1=(\eta+J(\eta))/2$ and $\eta_2=(\eta-J(\eta))/(2i)$ satisfy $J(\eta_1)=\eta_1$ and $J(\eta_2)=\eta_2$, this proves the claim.
Since $\mathcal H^1$ is dense in $\mathcal H^0$, Remark \ref{density} implies that $\mathcal K$ is dense $\mathcal F_M(\mathcal H)$. Using the  claim, we deduce that $\mathcal M\Omega$ is dense in $\mathcal F_M(\mathcal H)$. Hence, $\mathcal M$ is dense in $\text{L}^2(\Gamma(M,\mathcal H)'')$ and therefore $\mathcal M=\Gamma(M,\mathcal H)''$. This proves the main assertion.

Let $\mathcal H^2\subset\mathcal H^0$ be the span of $\{M\xi_iM\}_{i\in I}$. Then $J(\mathcal H^2)=\mathcal H^2$. Thus, the main assertion implies that $\Gamma(M,\mathcal H)''$ is equal to $\Big(M\cup\{s(\xi)\mid \xi\in \mathcal H^2, J(\xi)=\xi\}\Big)''$.
Denote $\mathcal N=\Big(M\cup\{s(\xi_i)\mid i\in I\}\Big)''$. Since $\ell(x\xi)=x\ell(\xi)$, $\ell(\xi x)=\ell(\xi) x$, and $J(x\xi x^*)=xJ(\xi) x^*$, for every $x\in M$ and $\xi\in\mathcal H^0$, we get that $s(y\xi_i y^*)=ys(\xi_i)y^*\in\mathcal N$, for every $y\in M$ and $i\in I$.  Note that $\{\xi\in\mathcal H^2\mid J(\xi)=\xi\}$ is equal to the linear span of $\{v\xi_i w^*+w\xi_i v^*\mid v,w\in M, i\in I\}$ and further that of $\{y\xi_iy^*\mid y\in M, i\in I\}$. Since $s(\xi_1+\xi_2)=s(\xi_1)+s(\xi_2)$, for every $\xi_1,\xi_2\in\mathcal H^0$, we conclude that $s(\xi)\in\mathcal N$, for every $\xi\in\mathcal H^2$ with $J(\xi)=\xi$. Since $\mathcal N$ also contains $M$, we get that $\mathcal M=\Gamma(M,\mathcal H)''$.
\end{proof}

An ultrafilter  on a directed set $(I,\leq)$ is called
{\it cofinal} if it contains $\{i\in I\mid i\geq i_0\}$, for every $i_0\in I$. 
If $(x_i)_{i\in I}\subset\mathbb C$ is a net such that $\lim_i x_i=x$ and $\mathcal U$ is a cofinal ultrafilter on $I$, then $\lim_{i\rightarrow\mathcal U}x_i=x$. 
We now arrive at the main result of this section.

\begin{lemma}\label{ultraproduct}
Let $(M,\tau)$ be a tracial von Neumann algebra, $\Phi:M\rightarrow M$ a subunital, symmetric completely positive map, and $\Phi_i:M\rightarrow M$, $i\in I$, a net of subunital, symmetric completely positive maps such that $\|\Phi_i(x)-\Phi(x)\|_2\rightarrow 0$, for every $x\in M$. Let $\mathcal U$ be a cofinal ultrafilter on $I$. Then there is a trace preserving $*$-homomorphism $\pi:\Gamma(M,\mathcal H_{\Phi})''\rightarrow\prod_{\mathcal U}\Gamma(M,\mathcal H_{\Phi_i})''$ such that $\pi_{|M}=\emph{Id}_M$ and $\pi(s(\xi_\Phi))=(s(\xi_{\Phi_i}))$.
Moreover, $\mathcal H_\Phi\subset \emph{L}^2(\prod_{\mathcal U}\Gamma(M,\mathcal H_{\Phi_i})'')\ominus\emph{L}^2(M^\mathcal U)$, as an $M$-bimodule.

\end{lemma}

In order to prove this result, we first introduce a definition and then establish a lemma.
\begin{definition}
Let $n\in\mathbb N$. 
 We define sets of formulas $S_{1}\subset\cdots\subset S_{n}$ involving non-commutative variables $X_0,X_1,\cdots,X_n$ and a function $\Psi$ which takes  values in an algebra and can be evaluated on all  polynomials in  $X_0,X_1,\cdots,X_n$,
as follows. Let $S_{1}=\{1,X_0,X_1\}$ and define inductively $S_{i}=\{X_{i}\}\cup S_{i-1}\cup\{X_i\Psi(a)b\mid a,b\in S_{i-1}\}$, for every $2\leq i\leq n$.
Given an algebra $M$, a function $\Phi:M\rightarrow M$ and an $(n+1)$-tuple ${\bf x}=(x_0,x_1,\cdots,x_n)\in M^{n+1}$, for every $f\in S_n$ we denote by $f^{\Phi}({\bf x})$
the element of $M$ obtained by replacing $X_0,X_1,\cdots,X_n$ with $x_0,x_1,\cdots,x_n$ and $\Psi$ with $\Phi$.

\end{definition}

\begin{lemma}\label{trace} There exist $\{\varepsilon(a)\mid a\in S_{n}\},\{\varepsilon(a_1,b_1,\cdots,a_k,b_k)\mid a_1,b_1,\cdots,a_k,b_k\in S_{n}\}\subset\mathbb Z$, for every $1\leq k\leq n$, such that the following holds.
 Let  $\Phi:M\rightarrow M$ be a normal, symmetric completely positive map, where $(M,\tau)$ is a tracial von Neumann algebra. Consider the associated  $M$-bimodule $(\mathcal H_{\Phi},\xi_{\Phi})$ 
and  the von Neumann algebra $\Gamma(M,\mathcal H_{\Phi})''\subset\mathbb B(\mathcal F_M(\mathcal H_{\Phi}))$. 
 Then for every  ${\bf x}=(x_0,x_1,\cdots,x_n)\in M^{n+1}$, the vector $x_ns(\xi_{\Phi})x_{n-1}\cdots x_1s(\xi_{\Phi})x_0\Omega\in\mathcal F_M(\mathcal H_{\Phi})$ is equal to
 $$\sum_{a\in S_{n}}\varepsilon(a)\; a^{\Phi}({\bf x})+ \sum_{k=1}^n\sum\limits_{\substack{a_0, a_1,b_1,\cdots,\\a_k,b_k\in S_{n}}}\varepsilon(a_0,a_1,b_1,\cdots,a_k,b_k)\; a_0^{\Phi}({\bf x}) (a_1^{\Phi}({\bf x})\xi_{\Phi}b_1^{\Phi}({\bf x}))\otimes\cdots\otimes (a_k^{\Phi}({\bf x})\xi_{\Phi}b_k^{\Phi}({\bf x})).$$
 In particular, $\tau(x_ns(\xi_{\Phi})x_{n-1}\cdots x_1s(\xi_{\Phi})x_0)=\sum_{a\in S_{n}}\varepsilon(a)\;\tau(a^{\Phi}({\bf x}))$.
\end{lemma}

\begin{proof}
Let $x,a,b,c\in M$ and $\eta\in\mathcal H_{\Phi}^{\otimes_M^m}$, for $m\geq 1$. Then $xs(\xi_{\Phi})(a)=x\xi_{\Phi}a$ and \eqref{Phi} implies that $$xs(\xi_{\Phi})(b\xi_{\Phi}c\otimes\eta)=x\xi_{\Phi}\otimes b\xi_{\Phi}c\otimes\eta+x\langle\xi_{\Phi},b\xi_{\Phi}c\rangle_M\eta=x\xi_{\Phi}\otimes b\xi_{\Phi}c\otimes\eta+(x\Phi(b)c)\eta.$$
The proof is now immediate by induction.
\end{proof}

\begin{proof}[Proof of Lemma \ref{ultraproduct}]
Let $n\in\mathbb N$ and  ${\bf x}=(x_0,x_1,\cdots,x_n)\in M^{n+1}.$ 
Since $\|\Phi_i(y)-\Phi(y)\|_2\rightarrow 0$, for every $y\in M$, it follows that $\|a^{\Phi_i}({\bf x})-a^{\Phi}({\bf x})\|_2\rightarrow 0$, for every $a\in S_n$. By combining this fact with Lemma \ref{trace}, we conclude that 
\begin{equation}\label{tau}
    \text{$ \lim_i\tau(x_ns(\xi_{\Phi_i})x_{n-1}\cdots x_1s(\xi_{\Phi_i})x_0)=\tau(x_ns(\xi_{\Phi})x_{n-1}\cdots x_1s(\xi_{\Phi})x_0)$,  $\forall x_0,x_1,\cdots,x_n\in M$}.
\end{equation}
Since $\Phi_i$ is subunital,  $\|s(\xi_{\Phi_i})\|\leq 2$, for every $i\in I$.
This implies that \begin{equation}\label{bound}(x_ns(\xi_{\Phi_i})x_{n-1}\cdots x_1s(\xi_{\Phi_i})x_0)\in\prod_{\mathcal U}\Gamma(M,\mathcal H_{\Phi_i})'', \forall x_0,x_1,\cdots,x_n\in M.
\end{equation}
Since $\text{span}(M\xi_\Phi M)$ is dense in $\mathcal H_{\Phi}$, Lemma \ref{density} implies that the linear span of $$\{x_ns(\xi_{\Phi})x_{n-1}\cdots x_1s(\xi_{\Phi})x_0\mid x_0,x_1,\cdots,x_n\in M\}$$ is $\|\cdot\|_2$
-dense in $\Gamma(M,\mathcal H_{\Phi})''$. This fact, \eqref{tau} and \eqref{bound} together imply  the existence of a trace preserving $*$-homomorphism 
$\pi:\Gamma(M,\mathcal H_{\Phi})''\rightarrow\prod_{\mathcal U}\Gamma(M,\mathcal H_{\Phi_i})''$ such that $$\pi(x_ns(\xi_{\Phi})x_{n-1}\cdots x_1s(\xi_{\Phi})x_0)=(x_ns(\xi_{\Phi_i})x_{n-1}\cdots x_1s(\xi_{\Phi_i})x_0)$$ and $\pi(x)=x$,  for every  $x_0,x_1,\cdots,x_n,x\in M$. This finishes the proof of the main assertion.

 Finally, the definition of $\pi$ implies that we have $\pi(Ms(\xi_\Phi)M)\subset\prod_\mathcal U\Gamma(M,\mathcal H_{\Phi_i})''\ominus M^{\mathcal U}$. Since the $M$-bimodule $\overline{\text{span}(Ms(\xi_\Phi)M)}$ is isomorphic to $\mathcal H_\Phi$, the moreover assertion follows.
\end{proof}

\section{Proofs of main results}\label{proofs}

\subsection{Proof of Theorem \ref{ec}}
We will prove (1) in the case of existentially closed separable II$_1$ factors $M$, and that (1) $\Rightarrow$ (2) $\Rightarrow$ (3) $\Leftrightarrow$ (4), for general separable II$_1$ factors $M$.

(1) 
Assume that $M$ is an existentially closed separable II$_1$ factor. By Remark \ref{symmetrize}, for the purpose of proving (1), after replacing $\mathcal H$ with $\mathcal H\oplus\overline{\mathcal H}$, we may assume $\mathcal H$ to be a symmetric $M$-bimodule.
Put $N:=\Gamma(M,\mathcal H)''$. Since  $M\subset N$, we can find an embedding $\pi:N\rightarrow M^{\mathcal U}$, for some ultrafilter $\mathcal U$ on a set $I$, whose restriction $\pi_{|M}$ is the diagonal embedding of $M$. Moreover, if $\mathcal H$ is separable, then so is $N$, and we can take $I=\mathbb N$.
Then $\pi$ extends to an embedding of Hilbert $M$-bimodules $\text{L}^2(N)\subset \text{L}^2(M^\mathcal U)$. Since $\mathcal H\subset \text{L}^2(N)$, part (1) follows.

 (1) $\Rightarrow$ (2) Assume that (1) holds. Let $\mathcal H$ be an $M$-bimodule. Then $\mathcal H\subset\text{L}^2(M^{\mathcal U})$, for a free ultrafilter $\mathcal U$ on a set $I$.
 Fix $
(\xi_j)_{j=1}^k\subset\mathcal H$ such that $\|\xi_j\|_2\leq 1$ for every $1\leq j\leq k$, a finite set $F\subset (M)_1$  and $\varepsilon>0$.  Then we can find $\eta_j\in M^{\mathcal U}$ such that $\|\eta_j\|_2\leq 1$ and $\|\xi_j-\eta_j\|_2<\frac{\varepsilon}{2}$, for every $1\leq j\leq k$. Then  $|\langle x\xi_j y,\xi_{j'}\rangle-\langle x\eta_j y,\eta_{j'}\rangle|<\varepsilon$, for every $1\leq j, j'\leq k$ and $x,y\in F$. 
Represent $\eta_j=(\eta_{i,j})$, where $\eta_{i,j}\in M$, for every $i\in I,1\leq j\leq k$. Since $\langle x\eta_j y,\eta_{j'}\rangle=\lim\limits_{i\rightarrow\mathcal U}\langle x\eta_{i,j} y,\eta_{i,j'}\rangle$, there is $i\in I$ such that $$\text{$|\langle x\xi_j y,\xi_{j'}\rangle-\langle x\eta_{i,j} y,\eta_{i,j'}\rangle|<\varepsilon$, for every $x,y\in F$ and $1\leq j,j'\leq k$.}$$ This implies that $\mathcal H$ belongs to the closure of 
$\text{L}^2(M)$ in the Fell topology, thus proving (2).

(2) $\Rightarrow$ (3) This implication is obvious.

(3) $\Rightarrow$ (4) Assume that (3) holds.
Let $\Phi:M\rightarrow M$ be a normal completely positive map. Then $\mathcal H_{\Phi}\subset_{\text{weak}}\text{L}^2(M)$. Since $\mathcal P_M$ is  the set of coefficients of $\text{L}^2(M)^{\oplus\infty}$ and $\mathcal S_M$ is the set of subunital, subtracial coefficients of $\text{L}^2(M)^{\oplus\infty}$, applying Lemma \ref{trivialbimod} gives (4).

(4) $\Rightarrow$ (3) Assume that (4) holds. 
If $\mathcal H$ is an $M$-bimodule, then $\mathcal H=\bigoplus_i\overline{\text{span}(M\xi_iM)}$, for a family of bounded vectors $(\xi_i)\subset \mathcal H$. Thus, in order to prove (3), it suffices to argue that any $M$-bimodule of the form $\mathcal H=\overline{\text{span}(M\xi M)}$, for some bounded vector $\xi\in\mathcal H$, is weakly contained  in $\text{L}^2(M).$ We denote by $\Phi:=\Phi_\xi:M\rightarrow M$ the associated normal completely positive map. Then $\mathcal H$ is isomorphic to $\mathcal H_{\Phi}$.
Since (4) holds, we can find a sequence $(\Phi_n)$ of coefficients of $\text{L}^2(M)^{\oplus\infty}$ such that $\|\Phi_n(x)-\Phi(x)\|_2\rightarrow 0$, for every $x\in M$. Thus, $\mathcal H_{\Phi}\subset_{\text{weak}}\text{L}^2(M)$ and so $\mathcal H\subset_{\text{weak}}\text{L}^2(M)$.
\hfill$\square$

 Next we prove Theorem \ref{bimod} and then use it to deduce Corollary \ref{hyper}.
  
 \subsection{Proof of Theorem \ref{bimod}}
Let $\mathcal H,\mathcal K$ be $M$-bimodules such that $\mathcal H\subset_{\text{weak}}\mathcal K$. Let  $(\xi_j)_{j\in J}\subset\mathcal H$ be a family of non-zero subtracial vectors such that denoting $\mathcal H_j=\overline{\text{span}(M\xi_jM)}$, we have $\mathcal H=\bigoplus_{j\in J}\mathcal H_j$. For $j\in J$, consider the subunital, subtracial completely positive map $\Phi_j:=\Phi_{\xi_j}:M\rightarrow M$. 
 Then  $\mathcal H_{\Phi_j}$ is isomorphic to $\mathcal H_j$ and so $\mathcal H_{\Phi_j}\subset_{\text{weak}}\mathcal K$. By Lemma \ref{trivialbimod}, we can find a net of subunital, subtracial coefficients $(\Phi_{i,j})_{i\in I}$ of $\mathcal K^{\oplus\infty}$ such that $\lim_i\|\Phi_{i,j}(x)-\Phi_j(x)\|_2\rightarrow 0$, for every $x\in M$. Here, $I$ is the set of pairs $(F,\varepsilon)$, with $F\subset M$ finite and $\varepsilon>0$, ordered by $(I,\varepsilon)\leq (I',\varepsilon')$ if and only if $I\subset I'$ and $\varepsilon\geq\varepsilon'$. If $M$ is separable, we can take $I=\mathbb N$. 
 Let $\mathcal U$ be a cofinal ultrafilter on $I$.

Since $\mathcal H_{\Phi_{i,j}}\subset\mathcal K^{\oplus\infty}$, we have 
 a natural trace preserving embedding $\Gamma(M,\mathcal H_{\Phi_{i,j}})''\subset\Gamma(M,\mathcal K^{\oplus\infty})''$.
By Lemma \ref{ultraproduct}, we find a trace preserving $*$-homomorphism $\pi_j:\Gamma(M,\mathcal H_j)''\rightarrow (\Gamma(M,\mathcal K^{\oplus\infty})'')^{\mathcal U}$ such that $(\pi_j)_{|M}=\text{Id}_M$.
Note that if $(M_j,\tau_j)_{j\in J}$ are tracial von Neumann algebras containing $(M,\tau)$ such that $\tau={\tau_j}_{|M}$, for every $j\in J$, then we have a natural embedding $*_{M,j\in J}M_j^{\mathcal U}\subset (*_{M,j\in J}M_j)^{\mathcal U}$. 
This implies the existence of a $*$
-homomorphism $$\pi=*_{M,j\in J}\pi_j:*_{M,j\in J}\Gamma(M,\mathcal H_j)''\rightarrow *_{M,j\in J}(\Gamma(M,\mathcal K^{\oplus\infty})'')^{\mathcal U}\subset(*_{M,j\in J}\Gamma(M,\mathcal K^{\oplus\infty})'')^{\mathcal U}$$ with $\pi_{|M}=\text{Id}_M$. 
By \cite[Proposition 2.18]{Sh97}, we get that  $*_{M,j\in J}\Gamma(M,\mathcal H_j)''=\Gamma(M,\mathcal H)''$ and $*_{M,j\in J}\Gamma(M,\mathcal K^{\oplus\infty})''=\Gamma(M,\bigoplus_{j\in J}\mathcal K^{\oplus\infty})''=\Gamma(M,\mathcal K\otimes\ell^2(S))''$, where $S=\mathbb N\times J$. This proves the main assertion.

Finally, assume that $M$ and $\mathcal H$ are separable.  Since $M$ is separable, we can take $I=\mathbb N$. Since $\mathcal H$ is separable, we can take $J=\mathbb N$. Altogether, this implies that we can take $I=S=\mathbb N$, proving the moreover assertion.
\hfill$\square$

 \subsection{Proof of Corollary \ref{hyper}}
Let $\mathcal H$ be an $R$-bimodule. By Remark \ref{symmetrize}, after replacing $\mathcal H$ with $\mathcal H\oplus\overline{\mathcal H}$, we may assume that $\mathcal H$ is symmetric. Then $\mathcal H\subset\text{L}^2(\Gamma(R,\mathcal H)'')$, as $R$-bimodules.
Let $\mathcal K=\text{L}^2(R)\otimes\text{L}^2(R)$. Then by \cite[Example 3.3(a)]{Sh97}, $\Gamma(R,\mathcal K)''=R*\text{L}(\mathbb Z)$. Further, applying \cite[Proposition 2.18]{Sh97} implies that $\Gamma(R,\mathcal K\otimes\ell^2(S))''=R*\text{L}(\mathbb F_S)$, for every set $S$.
Since $R$ is hyperfinite, we have that $\mathcal H\subset_{\text{weak}}\mathcal K$. Theorem \ref{bimod} thus provides a trace preserving $*$-homomorphism $\pi:\Gamma(R,\mathcal H)''\rightarrow (R*\text{L}(\mathbb F_S))^{\mathcal U}$ such that $\pi_{|R}=\text{Id}_R$, for a set $S$ and a cofinal ultrafilter $\mathcal U$
 on a set $I$. Thus, $\text{L}^2(\Gamma(R,\mathcal H)'')\subset\text{L}^2((R*\text{L}(\mathbb F_S))^{\mathcal U})$ and therefore
 \begin{equation}\label{eq1}
     \text{$\mathcal H\subset\text{L}^2((R*\text{L}(\mathbb F_S))^{\mathcal U})$, as $R$-bimodules.}
 \end{equation}
 To finish the proof, we will use the following fact (see \cite{CP09}):
 
 \begin{fact} \label{fact}
 Let $\mathcal V_j$ be an ultrafilter on a set $I_j$, for all $j\in\{1,2\}$. Let $\mathcal V_1\otimes\mathcal V_2$ be the ultrafilter on $I_1\times I_2$ defined as follows:  $\text{$\lim_{(i_1,i_2)\rightarrow\mathcal V_1\otimes\mathcal V_2}f(i_1,i_2)=\lim_{i_1\rightarrow\mathcal V_1}(\lim_{i_2\rightarrow\mathcal V_2}f(i_1,i_2))$, for every $f\in\ell^\infty(I_1\times I_2)$}.$ Then for every tracial von Neumann algebra $(N,\tau)$, we have a trace preserving $*$-isomorphism $N^{\mathcal V_1\otimes\mathcal V_2}\cong (N^{\mathcal V_2})^{\mathcal V_1}$ given by $x\mapsto ((x_{i_1,i_2})_{i_2\in I_2})_{i_1\in I_1}$, for $x=(x_{i_1,i_2})_{(i_1,i_2)\in I_1\times I_2}\in\ell^{\infty}(I_1\times I_2,N)$.

 \end{fact}
 Let $J$ be the collection of finite subsets $T\subset S$ ordered by inclusion, and $\mathcal V$ a cofinal ultrafilter on $J$. Then the map $R*\text{L}(\mathbb F_S)\ni x\mapsto (\text{E}_{R*\text{L}(\mathbb F_T)}(x))_{T\in J}\in\prod_{\mathcal V}(R*\text{L}(\mathbb F_T))$ is a trace preserving $*$-homomorphism. For every finite subset $T\subset S$, view $\mathbb F_T$ as a subgroup of $\mathbb F_\infty$.  By combining the last two facts, we get a trace preserving $*$-homomorphism $\delta:R*\text{L}(\mathbb F_S)\rightarrow (R*\text{L}(\mathbb F_\infty))^{\mathcal V}$ such that $\delta_{|R}=\text{Id}_R.$ Since $\text{L}(\mathbb F_\infty)$ is Connes embeddable, we can find a trace preserving $*$-homomorphism $\rho:R*\text{L}(\mathbb F_\infty)\rightarrow R^{\mathcal W}$, with $\mathcal W$ a free ultrafilter on $\mathbb N$. Since any embedding of $R$ into $R^{\mathcal W}$ is unitarily conjugate to the diagonal embedding, we may assume that $\rho_{|R}=\text{Id}_R$. 
 Thus, $\rho^{\mathcal V}:(R*\text{L}(\mathbb F_\infty))^{\mathcal V}\rightarrow (R^{\mathcal W})^{\mathcal V}$ given by $\rho^{\mathcal V}((x_j)_{j\in J})=(\rho(x_j))_{j\in J}$ is a trace preserving $*$-homomorphism.  
 
 Define $\sigma:=(\rho^{\mathcal V}\circ\delta)^{\mathcal U}:(R*\text{L}(\mathbb F_S))^{\mathcal U}\rightarrow ((R^{\mathcal W})^{\mathcal V})^{\mathcal U}.$
 Using Fact \ref{fact}, we can view $\sigma$ as trace preserving $*$-homomorphism $\sigma:(R*\text{L}(\mathbb F_S))^{\mathcal U}\rightarrow R^{\mathcal U\otimes\mathcal V\otimes W}$. Moreover, it is immediate that $\sigma_{|R}=\text{Id}$. In particular, we derive that
 \begin{equation}\label{eq2}
   \text{$ \text{L}^2((R*\text{L}(\mathbb F_S))^{\mathcal U})\subset\text{L}^2(R^{\mathcal U\otimes\mathcal V\otimes W})$, as $R$-bimodules}.
 \end{equation}
 By combining \eqref{eq1} and \eqref{eq2}, we conclude that $\mathcal H\subset\text{L}^2(R^{\mathcal U\otimes\mathcal V\otimes\mathcal W})$, as $R$-bimodules.
 This proves the main assertion.

 Finally, assume that $\mathcal H$ is separable. The moreover part of Theorem \ref{bimod} implies that we can take $S=I=\mathbb N$. Then $J$ is countable. Thus, $\mathcal U,\mathcal V$, $\mathcal W$, and consequently $\mathcal U\otimes\mathcal V\otimes\mathcal W$ are ultrafilters on countable sets. This implies the moreover assertion.
   \hfill$\square$

   We next use Theorem \ref{bimod} and its proof to derive Corollary \ref{relamen}.
   \subsection{Proof of Corollary \ref{relamen}}

By Remark \ref{symmetrize}, in order to prove (1) and (2), we may assume that $\mathcal H$ is a symmetric $M$-bimodule.

   (1) Assume that $\mathcal H$ is a symmetric $M$-bimodule with $\mathcal H\subset_{\text{weak}}\text{L}^2(M)$.
   By combining Proposition 2.18 and Example 3.3(b) in \cite{Sh97}, we get that $\Gamma(M,\text{L}^2(M)\otimes\ell^2(\mathbb N))''=\widehat{M}$, where $\widehat{M}=M\overline{\otimes}\text{L}(\mathbb F_\infty)$.
   By Theorem \ref{bimod}, there exists a trace preserving $*$-homomorphism $\pi:\Gamma(M,\mathcal H)''\rightarrow\widehat{M}^{\mathcal U}$, where $\mathcal U$ is an ultrafilter on $\mathbb N$, such that $\pi_{|M}=\text{Id}_M$. This implies that 
   $\mathcal H\subset \text{L}^2(\Gamma(M,\mathcal H)'')\subset\text{L}^2(\widehat{M}^\mathcal U)$, as $M$-bimodules.

   Conversely, if $\mathcal H\subset\text{L}^2(\widehat{M}^\mathcal U)$, then $\mathcal H\subset_{\text{weak}}\text{L}^2(\widehat{M})=\text{L}^2(M)\otimes\ell^2(\mathbb N)$ and therefore $\mathcal H\subset_{\text{weak}}\text{L}^2(M)$.

(2) Assume that $\mathcal H$ is a symmetric $M$-bimodule with $\mathcal H\subset_{\text{weak}}\text{L}^2(M)\otimes\text{L}^2(M)$.
   By combining Proposition 2.18 and Example 3.3(a) in \cite{Sh97}, we get that $\Gamma(M,(\text{L}^2(M)\otimes\text{L}^2(M))\otimes\ell^2(\mathbb N))''=\overline{M}$, where $\overline{M}=M*\text{L}(\mathbb F_\infty)$.
   By Theorem \ref{bimod}, there exists a $*$-homomorphism $\pi:\Gamma(M,\mathcal H)''\rightarrow\overline{M}^{\mathcal U}$ such that $\pi_{|M}=\text{Id}_M$. This implies that 
   $\mathcal H\subset \text{L}^2(\Gamma(M,\mathcal H)'')\subset\text{L}^2(\overline{M}^\mathcal U)$, as $M$-bimodules. Moreover, a close inspection of the proof of Theorem \ref{bimod} (see the moreover assertion of Lemma \ref{ultraproduct}) shows that  $\mathcal H\subset\text{L}^2(\overline{M}^\mathcal U)\ominus\text{L}^2(M^\mathcal U)$, as desired.

   Conversely, if $\mathcal H\subset\text{L}^2(\overline{M}^\mathcal U)\ominus\text{L}^2(M^{\mathcal U})$, then $\mathcal H\subset_{\text{weak}}\text{L}^2(\overline{M})\ominus\text{L}^2(M)$. On the other hand, we have $\text{L}^2(\overline{M})\ominus\text{L}^2(M)\cong(\text{L}^2(M)\otimes\text{L}^2(M))\otimes\ell^2(\mathbb N)$, and thus we conclude that $\mathcal H\subset_{\text{weak}}\text{L}^2(M)\otimes\text{L}^2(M)$.

To prove (3) and (4), let $\mathcal K=\text{L}^2(M)\otimes_Q\text{L}^2(M)$ and recall that we defined  $\widetilde M=M*_Q(Q\overline{\otimes}\text{L}(\mathbb Z))$. Then $\widetilde M \cong \Gamma(M,\mathcal K)''$ by \cite[Example 3.3(c)]{Sh97}.

 (3)
If $P$ is amenable relative to $Q$ inside $M$, by \cite[Proposition 2.4]{PV11} we find a net of subtracial vectors $(\xi_n)\subset \mathcal K_{+}$ such that we have $\lim_n\|\langle\xi_n,\xi_n\rangle_M-\tau(\cdot)\|_1=0$ and $\lim_n\|y\xi_n-\xi_ny\|=0$, for every $y\in P$. Moreover, since $M$ is separable, we can take $(\xi_n)$ to be a sequence.

Define $x_n=s(\xi_n)\in \widetilde M$. Since $\|s(\xi_n)\|\leq 2 \|\langle\xi_n,\xi_n\rangle_M\|\leq 2$ and $\text{E}_M(s(\xi_n))=0$, for every $n$, we get that $x=(x_n)$ belongs to ${\widetilde M}^{\mathcal U}$ and $\text{E}_{M^{\mathcal U}}(x)=0$.  If $\Omega=\widehat{1}\in\mathcal F_M(\mathcal K)$ is the vacuum vector, then $(yx_n-x_ny)\Omega=y\xi_n-\xi_n y$ and thus $\lim_n\|yx_n-x_ny\|_2=0$, for every $y\in P$. This implies that $x\in P'\cap \widetilde M^{\mathcal U}$. Since $\langle\xi_n,\xi_n\rangle_M\leq 1$, we get that $\|\langle \xi_n,\xi_n\rangle_M-1\|_2\leq 2\|\langle \xi_n,\xi_n\rangle_M-1\|_1^{1/2}$ and thus $\|\langle\xi_n,\xi_n\rangle_M-1\|_2\rightarrow 0$. Since $\text{E}_M(x_n^*x_n)=\langle \xi_n,\xi_n\rangle_M$, we get that $\text{E}_{M^\mathcal U}(x^*x)=(\text{E}_M(x_n^*x_n))=1$. This proves the only if assertion.

Conversely, assume that there exists $x\in P'\cap\widetilde M^{\mathcal U}$ such that $\text{E}_{M^\mathcal U}(x)=0$ and $\text{E}_{M^{\mathcal U}}(x^*x)=1$. Then we can find a net $(x_n)\subset \widetilde M\ominus M$ such that $\sup\|x_n\|<\infty$, $\lim_n\|\text{E}_M(x_n^*x_n)-1\|_2=0$ and $\lim_n\|yx_n-x_ny\|_2=0$, for every $y\in P$.
This gives that $\tau(x_n^* yx_nz)=\tau(yz)$, for every $y\in P$ and $z\in M$.
Thus, the $P$-$M$-bimodule $\text{L}^2(M)$ is weakly contained in $\text{L}^2(\widetilde M)\ominus\text{L}^2(M)$.
On the other hand,  \cite[Lemma 2.10 (1)]{Io12} gives that $\text{L}^2(\widetilde M)\ominus\text{L}^2(M)=\mathcal L\otimes_Q\text{L}^2(M)$, as $M$-bimodules, for an $M$-$Q$ bimodule $\mathcal L$. By \cite[Proposition 2.4]{PV11},  $P$ is amenable relative to $Q$ inside $M$.

(4) Assume that there exists a  $*$-homomorphism $\pi:\widetilde M\rightarrow {\widehat{M}}^{\mathcal U}$ whose restriction to $M$ is the diagonal embedding of $M$. Then $\text{L}^2(\widetilde M)\subset\text{L}^2(\widehat{M}^{\mathcal U})$, as $M$-bimodules. If $u\in\text{L}(\mathbb Z)\subset\widetilde M$ is a Haar unitary, then the $M$-bimodule $\overline{\text{span}(MuM)}$ is isomorphic to $\text{L}^2(M)\otimes_Q\text{L}^2(M)$. Thus, $\text{L}^2(M)\otimes_Q\text{L}^2(M)\subset\text{L}^2(\widehat{M}^\mathcal U)$, which implies that $\text{L}^2(M)\otimes_Q\text{L}^2(M)\subset_{\text{weak}}\text{L}^2(\widehat{M})=\text{L}^2(M)\otimes\ell^2(\mathbb N)$. Hence $\text{L}^2(M)\otimes_Q\text{L}^2(M)\subset_{\text{weak}}\text{L}^2(M)$, as claimed.

Conversely, assume that $\mathcal K\subset_{\text{weak}}\text{L}^2(M)$. Since $\Gamma(M,\mathcal K)''=\widetilde M$, $\Gamma(M,\text{L}^2(M)\otimes\ell^2(\mathbb N))''=\widehat{M}$ and $\mathcal K=\mathcal H_{\text{E}_Q}$, by combining Lemma \ref{trivialbimod} and Lemma \ref{ultraproduct} we deduce the existence of a trace preserving $*$-homomorphism $\pi:\widetilde{M}\rightarrow\widehat{M}^\mathcal U$ such that $\pi_{|M}=\text{Id}_M$.
\hfill$\square$

\subsection{Proof of Theorem \ref{construction}}
Let $M_0$ be a separable II$_1$ factor. We will construct an existentially closed separable II$_1$ factor $M\supset M_0$. To this end,  let $X=\{x_k\mid k\in\mathbb N\}\subset (M_0)_1$ be a sequence which generates $M_0$ and $\mathcal U$ be a free ultrafilter on $\mathbb N$. The proof relies on the following claim:

\begin{claim}\label{inductive}
There exists a separable II$_1$ factor $M_1\supset M_0$ such that the following holds: for any separable tracial von Neumann algebra $N\supset M_0$, there is an embedding of $N$ into $M_1^{\mathcal U}$ whose restriction to $M_0$ is the diagonal embedding.
\end{claim}

\begin{proof}[Proof of Claim \ref{inductive}]
We denote by $\mathcal F$ the set of pairs $(N,Y)$, where $N$ is a separable tracial von Neumann algebra which contains $M_0$ and $Y = X \cup \{y_k\mid k\in\mathbb N\}\subset (N)_1$ is a  sequence which generates $N$ and contains $X$.
For $m\in\mathbb N$, we denote by $\mathcal P_m$ the set of $*$-monomials in the variables $\{X_1,\cdots,X_m,Y_1,\cdots,Y_m\}$ of degree at most $m$.
For $(N,Y)\in\mathcal F$, $m \in \mathbb{N}$ and $\varepsilon >0$, we denote by $\mathcal{U}_{(N,Y)}(m,\varepsilon) $ the set of  
$(\bar{N},\bar{Y})\in\mathcal F$ such that writing $\bar{Y} = X \cup \{\bar{y}_k\mid k\in\mathbb N\}$ we have:
\[
\sum_{p\in\mathcal P_m}
|\tau_{\bar{N}}(p(x_1,...,x_m,\bar{y}_1,...,\bar{y}_m)) - \tau_{N}(p(x_1,...,x_m,y_1,...,y_m))| < \varepsilon.
\]
Consider the topology on $\mathcal F$ which has $\{\mathcal U_{(N,Y)}(m,\varepsilon)\}_{m\in\mathbb N,\varepsilon>0}$ as a neighborhood basis of any $(N,Y)\in\mathcal F$. Note that the elements of $\mathcal F$ can be viewed as representatives of non-commutative laws on infinitely many variables;  the topology we defined on $\mathcal F$ is precisely the weak$^*$-topology on the space of laws.
Since $\mathcal F$ is separable in this topology, it admits a dense sequence
$\{(N_n,Y_n)\}_{n\in\mathbb{N}}$. Define $M_1 = *_{M_0, n\in\mathbb N}N_n$. Write $Y_n=X\cup\{y_{n,k}\mid k\in\mathbb N\}$, for every $n\in\mathbb N$.

Let  $N$ be a separable tracial von Neumann algebra which contains $M_0$. Choose $Y\subset (N)_1$ such that
$(N,Y) \in \mathcal{F}$. Then $(N,Y)$ is the limit of a subsequence $(N_{n_l},Y_{n_l})$ of $(N_n,Y_n)$. Thus, we have $$\tau_{N_{n_l}}(p(x_1,...,x_m, y_{n_l,1},...,y_{n_l,m}))\rightarrow\tau_N(p(x_1,...,x_m,y_1,...,y_m)),\;\;\text{as $l\rightarrow\infty$},$$ for every $m\in\mathbb N$ and $*$-monomial $p$ in $X_1,...,X_m,Y_1,...,Y_m$.  

Since $\cup_{n\in\mathbb N}Y_n\subset M_1$,  Lemma \ref{hom} gives a trace preserving $*$-homomorphism $\pi:N\rightarrow M_1^{\mathcal U}$ such that $\pi(p(x_1,...,x_m,y_1,...,y_m))=(p(x_1,...,x_m, y_{n_l,1},...,y_{n_l,m}))$, for every $m\in\mathbb N$ and $*$-monomial $p$ in $X_1,...,X_m,Y_1,...,Y_m$. In particular, $\pi(x_k)=x_k$, for every $k\in\mathbb N$, and thus the restriction of $\pi$ to $M_0$ is the diagonal embedding of $M_0$ into $M_1^{\mathcal U}$. Moreover, since this property still holds if we replace $M_1$ by $M_1*\text{L}(\mathbb Z)$, we can assume that $M_1$ is a II$_1$ factor. Here we are using the fact that $P*\text{L}(\mathbb Z)$ is a II$_1$ factor for any non-trivial tracial von Neumann algebra $P$.
\end{proof}

By Claim \ref{inductive}, we can inductively construct an increasing sequence of II$_1$ factors $M_n\supset M_0$, $n\geq 1$,  such that for every $n\geq 0$ and any separable tracial von Neumann algebra $N\supset M_n$, there is an embedding of $N$ into $M_{n+1}^{\mathcal U}$ whose restriction to $M_n$ is the diagonal embedding.

Let $M = ({{\bigcup_{n\in\mathbb N} M_n})}^{''}$. Then $M\supset M_0$ is a separable II$_1$ factor.
We claim that $M$ is existentially closed. 
Let $N$ be a II$_1$ factor containing $M$. Let $\{z_k\mid k\in\mathbb N\}\subset (M)_1$ and $\{t_m\mid m\in\mathbb N\}\subset (N)_1$ be sequences which generate $M$ and $N$, respectively. We may assume that $z_k\in M_k$, for every $k\in\mathbb N$.  


Let $n\in\mathbb N$ and denote by $P$ the von Neumann subalgebra of $N$ generated by $M$ and $\{t_1,\cdots,t_n\}$.
Since $M_n\subset M\subset P$ and $P$ is separable there is a trace preserving $*$-homomorphism $\pi:P\rightarrow M_{n+1}^{\mathcal U}$ such that $\pi(x)=x$, for every $x\in M_n$.
In particular,  $\pi(z_k)=z_k$, for every $1\leq k\leq n$. Write $\pi(t_m)=(t_{l,m})$, where $t_{l,m}\in (M_{n+1})_1$, for every $l\in\mathbb N$ and $1\leq m\leq n$.
Since $\pi$ is trace preserving, for every $*$-polynomial $p$ in variables $Z_k,T_m, 1\leq k,m\leq n$, we have $$\lim_{l\rightarrow\mathcal U}\tau_{M}(p(z_k,t_{l,m}, 1\leq k,m\leq n))=\tau_N(p(z_k,t_m,1\leq k,m\leq n)).$$
This implies that we can find $l\in\mathbb N$ such that denoting $u_{n,m}:=t_{l,m}$, for every $1\leq m\leq n$, we have $$|\tau_{M}(p(z_k,u_{n,m}, 1\leq k,m\leq n))-\tau_N(p(z_k,t_m,1\leq k,m\leq n))|<\frac{1}{n},$$ for every $*$-monomial $p$ of degree $\leq n$ in variables $Z_k,X_m, 1\leq k,m\leq n$.

Then $\lim_n\tau_M(p(z_k,u_{n,m}, k,m\in\mathbb N))=\tau_N(p(z_k,t_m,k,m\in\mathbb N)),$ for every $*$-monomial $p$ in variables $Z_k,X_m, k,m\in\mathbb N$.
Since the set $\{z_k\mid k\in\mathbb N\}\cup\{t_m\mid m\in \mathbb N\}$ generates $N$, Lemma \ref{hom} gives a trace preserving $*$-homomorphism $\rho:N\rightarrow M^{\mathcal V}$ such that $\rho(z_k)=z_k$ and $\rho(t_m)=(u_{n,m})$, for every $k,m\in\mathbb N$. Since the set $\{z_k\mid k\in\mathbb N\}$ generates $M$, the restriction of $\rho$ to $M$ is the diagonal embedding into $M^{\mathcal V}$. Thus, $M$ is existentially closed in $N$, which finishes the proof. \hfill$\square$

\subsection{Proof of Theorem \ref{non-Gamma}}
Our next goal is to prove Theorem \ref{non-Gamma}. We first establish the following lemma.

\begin{lemma}\label{AFP}
Let $(M_1,\tau_1),(M_2,\tau_2)$ be tracial von Neumann algebras with a common von Neumann subalgebra $(B,\tau)$. Assume that $\tau={\tau_i}|B$ and $\emph{L}^2(M_i)\ominus\emph{L}^2(B)\subset_{\emph{weak}}\emph{L}^2(B)\otimes\emph{L}^2(B)$, as $B$-bimodules, for every $i\in\{1,2\}$. Then  $\emph{L}^2(M_1*_BM_2)\ominus\emph{L}^2(B)\subset_{\emph{weak}}\emph{L}^2(B)\otimes\emph{L}^2(B)$, as $B$-bimodules.
\end{lemma}

\begin{proof}
If two $B$-bimodules $\mathcal H_1$ and $\mathcal H_2$ are  weakly contained in $\text{L}^2(B)\otimes\text{L}^2(B)$, then so is $\mathcal H_1\otimes_B\mathcal H_2$.
Since the $B$-bimodule $\text{L}^2(M_1*_BM_2)\ominus\text{L}^2(B)$ decomposes as
$$\bigoplus_{n\geq 1}\bigoplus_{\substack{{i_1,i_2,\cdots,i_n\in\{1,2\}}\\{ i_1\not= i_2\not=\cdots\not=i_n}}}(\text{L}^2(M_{i_1})\ominus \text{L}^2(B))\otimes_B\cdots\otimes_B(\text{L}^2(M_{i_n})\ominus\text{L}^2(B)),$$ the conclusion follows.
\end{proof}

{\it Proof of Theorem \ref{non-Gamma}.}
Let $M_0$ be a separable II$_1$ factor. After replacing $M_0$ by $M_0*\text{L}(\mathbb Z)$, we may assume that $M_0$ is non-Gamma.
We will construct a non-Gamma separable II$_1$ factor $M$ containing $M_0$ such that $M$  is existentially closed in any II$_1$ factor $N$ which contains $M$ and satisfies that $\text{L}^2(N)\ominus\text{L}^2(M)\subset_{\text{weak}}\text{L}^2(M)\otimes\text{L}^2(M)$, as $M$-bimodules. Let $\mathcal U$ be a free ultrafilter on $\mathbb N$.

The rest of the proof follows closely the proof of Theorem \ref{construction}, whose notation we keep here. Let $\mathcal G$ be the subset of $\mathcal F$ consisting of $(N,Y)\in\mathcal F$ such that  $\text{L}^2(N)\ominus\text{L}^2(M_0)\subset_{\text{weak}}\text{L}^2(M_0)\otimes\text{L}^2(M_0)$ as $M_0$-bimodules.
Since $\mathcal G$ is separable in the topology defined in the proof of Theorem \ref{construction}, it admits a dense sequence
$\{(N_n,Y_n)\}_{n\in\mathbb{N}}$. Define $M_1 = (*_{M_0,n\in\mathbb N}N_n)*\text{L}(\mathbb Z)$.  Then $M_1$ is a II$_1$ factor. Using Lemma \ref{AFP} and  induction implies that
  $\text{L}^2(M_1)\ominus\text{L}^2(M_0)\subset_{\text{weak}}\text{L}^2(M_0)\otimes\text{L}^2(M_0)$, as  $M_0$-bimodules.

The proof of Claim \ref{inductive} shows that for every separable
tracial von Neumann algebra $N\supset M_0$ such that $\text{L}^2(N)\ominus\text{L}^2(M_0)\subset_{\text{weak}}\text{L}^2(M_0)\otimes\text{L}^2(M_0)$ as $M_0$-bimodules, there is a trace preserving  $*$-homomorphism $\pi:N\rightarrow M_1^{\mathcal U}$ whose restriction to $M_0$ is the diagonal embedding of $M_0$ into $M_1^{\mathcal U}$.

We then inductively construct an increasing sequence of II$_1$ factors $M_n$, $n\geq 1$, containing $M_0$ such that for every $n\geq 0$ the following properties hold:
\begin{enumerate}
\item\label{1} For any separable tracial von Neumann algebra $N$ which contains $M_n$ and satisfies that $\text{L}^2(N)\ominus\text{L}^2(M_n)\subset_{\text{weak}}\text{L}^2(M_n)\otimes\text{L}^2(M_n)$, as $M_n$-bimodules, there is  a trace preserving $*$-homomorphism $\pi:N\rightarrow M_{n+1}^{\mathcal U}$ whose restriction to $M_n$ is the diagonal embedding of $M_n$. 
\item\label{2}  $\text{L}^2(M_{n+1})\ominus\text{L}^2(M_n)\subset_{\text{weak}}\text{L}^2(M_n)\otimes\text{L}^2(M_n)$, as  $M_n$ bimodules.

\end{enumerate}
Then $M = ({{\bigcup_{n\in\mathbb N} M_n})}^{''}$ is a separable II$_1$ factor which contains $M_0$.
We will prove that $M$ has the desired properties.

First, let   $N\supset M$ be a II$_1$ factor with $\text{L}^2(N)\ominus\text{L}^2(M)\subset_{\text{weak}}\text{L}^2(M)\otimes\text{L}^2(M)$, as $M$-bimodules.
We claim that $M$ is existentially closed in $N$. Let $\{t_m\mid m\in\mathbb N\}\subset (N)_1$ be a sequence which  generates $N$.
Let $n\in\mathbb N$ and denote by $P\subset N$ the von Neumann subalgebra generated by $M$ and $\{t_1,\cdots,t_n\}$. 
By \eqref{2}, for every $k\geq n$ we have  $\text{L}^2(M_{k+1})\ominus\text{L}^2(M_k)\subset_{\text{weak}}\text{L}^2(M_n)\otimes\text{L}^2(M_n)$, as  $M_n$-bimodules.
Since we also have $\text{L}^2(M)\ominus\text{L}^2(M_n)=\bigoplus_{k\geq n}(\text{L}^2(M_{k+1})\ominus\text{L}^2(M_k))$, we get $\text{L}^2(M)\ominus\text{L}^2(M_n)\subset_{\text{weak}}\text{L}^2(M_n)\otimes\text{L}^2(M_n)$, as $M_n$-bimodules. From this we further derive that $\text{L}^2(N)\ominus\text{L}^2(M_n)\subset_{\text{weak}}\text{L}^2(M_n)\otimes\text{L}^2(M_n)$, and hence $\text{L}^2(P)\ominus\text{L}^2(M_n)\subset_{\text{weak}}\text{L}^2(M_n)\otimes\text{L}^2(M_n)$, as $M_n$-bimodules. We can now apply \eqref{1} to get  a trace preserving $*$-homomorphism $\pi:P\rightarrow M_{n+1}^{\mathcal U}$ whose restriction to $M_n$ is the diagonal embedding of $M_n$ into $M_{n+1}^{\mathcal U}$.
Proceeding as in the proof of Theorem \ref{construction} gives that $M$ is existentially closed in $N$.

Second, assume by contradiction that $M$ has property Gamma.
Then $M'\cap M^{\mathcal U}$ is diffuse, hence there is a unitary $u\in M'\cap M^{\mathcal U}$ of trace zero. Since $M_0$ is non-Gamma, $M_0'\cap M_0^{\mathcal U}=\mathbb C1$ and thus we get that $u\in M^{\mathcal U}\ominus M_0^{\mathcal U}$. Since $u$ commutes with $M_0$, we deduce that $\text{L}^2(M_0)\subset_{\text{weak}}\text{L}^2(M)\ominus\text{L}^2(M_0)$, as $M_0$-bimodules.
On the other hand, \eqref{2} implies that $\text{L}^2(M)\ominus\text{L}^2(M_0)\subset_{\text{weak}}\text{L}^2(M_0)\otimes\text{L}^2(M_0)$, as $M_0$-bimodules. Combining these facts gives that $\text{L}^2(M_0)\subset_{\text{weak}}\text{L}^2(M_0)\otimes\text{L}^2(M_0)$, as $M_0$-bimodules. In other words, $M_0$ is amenable, which contradicts the fact that $M_0$ is non-Gamma. \hfill$\square$

We continue by proving the following result mentioned 
in Remark \ref{optimal}. 
\begin{lemma}\label{freegroup}
Let $M=\emph{L}(\mathbb F_n)$, for some $n\geq 1$, and $\mathcal U$ be a free ultrafilter on $\mathbb N$. Then we have that $\emph{L}^2(M^{\mathcal U})\ominus\emph{L}^2(M)\subset_{\emph{weak}}\emph{L}^2(M)\otimes\emph{L}^2(M)$, as $M$-bimodules.
\end{lemma}

\begin{proof}
    The proof uses Popa's malleable deformation of $M$ into $\widetilde M=M*M$. More precisely, by \cite{Po86,Po06} there exist automorphisms $(\alpha_t)_{t\in\mathbb R}$ of $\widetilde M$ such that
    \begin{enumerate} 
    \item $\lim_{t\rightarrow 0}\|\alpha_t(x)-x\|_2=0$, for every $x\in  M$, and 
    \item the map $M\ni x\mapsto \text{E}_M(\alpha_t(x))\in M$ extends to a compact operator on $\text{L}^2(M)$, for any $t\not=0$.
    \end{enumerate}
    
    Let $F\subset (M)_1$ be a finite set, $\xi\in\text{L}^2(M^{\mathcal U})\ominus \text{L}^2(M)$ with $\|\xi\|_2\leq 1$, and $\varepsilon>0$. Let $\eta\in M^{\mathcal U}\ominus M$ such that $\|\xi-\eta\|_2<\frac{\varepsilon}{4}$ and $\|\eta\|_2\leq 1$.
    Then \begin{equation}\label{approxim} \text{$|\langle x\eta y,\eta\rangle-\langle x\xi y,\xi\rangle|< \frac{\varepsilon}{2}$, for every $x,y\in F$.}
    \end{equation}
  By (1), we can find $t>0$  such that $\|\alpha_t(x)-x\|_2<\frac{\varepsilon}{4(\|\eta\|+1)}$, for every $x\in F$. Since $F\subset (M)_1$ and $\|\eta\|_2\leq 1$, we get that
   \begin{equation}\label{approx2}
       \text{$|\langle\alpha_t(x)\eta\alpha_t(y),\eta\rangle-\langle x\eta y,\eta\rangle|\leq (\|\alpha_t(x)-x\|_2+\|\alpha_t(y)-y\|_2)\|\eta\|<\frac{\varepsilon}{2}$, for every $x,y\in F$.}
   \end{equation}
   Write $\eta=(\eta_n)$, where $(\eta_n)\subset M$ is a sequence with $\sup\|\eta_n\|<\infty$. Since $\eta\in M^{\mathcal U}\ominus M$, we have that $\lim_{n\rightarrow\mathcal U}\eta_n=0$, weakly. Then (2) implies that $\lim_{n\rightarrow\mathcal U}\|\alpha_{-t}(\eta_n)\|_2=0$.
   Thus, denoting $\zeta_n=\alpha_{-t}(\eta_n)-\text{E}_M(\alpha_{-t}(\eta_n))\in\widetilde M\ominus M$ we derive that for every $x,y\in M$ we have that
   \begin{equation}\label{approx3}
\text{$\langle\alpha_t(x)\eta\alpha_t(y),\eta\rangle=
\lim_{n\rightarrow\mathcal U}\langle\alpha_t(x)\eta_n\alpha_t(y),\eta_n\rangle=\lim_{n\rightarrow\mathcal U}\langle x\alpha_{-t}(\eta_n)y,\alpha_{-t}(\eta_n)\rangle=\lim_{n\rightarrow\mathcal U}\langle x\zeta_ny,\zeta_n\rangle$.}
   \end{equation}
   By combining \eqref{approxim}, \eqref{approx2} and \eqref{approx3}, we deduce that there is $n\in\mathbb N$ such that $\zeta=\zeta_n\in \widetilde M\ominus M$ satisfies $|\langle x\zeta y,\zeta\rangle-\langle x\xi y,\xi\rangle|<\varepsilon$, for every $x,y\in F$. Since the $M$-bimodule $\text{L}^2(\widetilde M)\ominus\text{L}^2(M)$ is isomorphic to $(\text{L}^2(M)\otimes\text{L}^2(M))\otimes\ell^2(\mathbb N)$, the conclusion follows.
\end{proof}

\begin{remark}
We are grateful to Changying Ding and Jesse Peterson for informing us that in a forthcoming paper they define a notion of biexactness for tracial von Neumann algebras $M$, and prove that the condition $\text{L}^2(M^{\mathcal U})\ominus\text{L}^2(M)\subset_{\text{weak}}\text{L}^2(M)\otimes\text{L}^2(M)$ from Lemma \ref{freegroup} is equivalent to the W$^*$AO property defined in \cite[Definition 2.1]{Ca22} and is implied by biexactnenss.
\end{remark}

\subsection{Proof of Theorem \ref{hap}}


Let $M$ be a separable II$_1$ factor. 

Assume first that $M$ has Haagerup's property. By applying \cite[Theorem 9]{OOT15} or \cite[Theorem 3.4]{BF07} we derive the existence of a strictly mixing $M$-bimodule $\mathcal M$ such that $\text{L}^2(M)\subset_{\text{weak}}\mathcal H$.  
By using Lemma \ref{subtracial} and after replacing $\mathcal H$ with $\mathcal H^{\oplus\infty}$, we get a strictly mixing $M$-bimodule $\mathcal H$ admitting a sequence of subtracial vectors $(\eta_n)$ such that $\lim_n\langle x\eta_n y, \eta_n\rangle=\tau(xy)$, for every $x,y\in M$. 
Next, note that the $M$-bimodule $\mathcal K:=\mathcal H\otimes_M\overline{\mathcal H}$ is symmetric, as witnessed by the involution $J(\xi\otimes_M\overline{\zeta})=\zeta\otimes_M\overline{\xi}$, and strictly mixing (see \cite[Proposition 7]{OOT15}). Then the vectors $\xi_n:=\eta_n\otimes_M\overline{\eta}_n\in\mathcal K$ are subtracial and satisfy $\lim_n\langle x\xi_n y, \xi_n\rangle=\tau(xy)$, for every $x,y\in M$. 

Let $\widetilde M=\Gamma(M,\mathcal K)''$. Then $\widetilde M$ is a factor (see \cite{KV15}, Theorem 5.1). We will prove that $\widetilde M$ has the desired properties. First, note that the $M$-bimodule $\text{L}^2(\widetilde M)\ominus\text{L}^2(M)$  is isomorphic to $\mathcal F_M(\mathcal K)\ominus\text{L}^2(M)$ and thus to $\mathcal K\otimes_M\mathcal F_M(\mathcal K)$. Since $\mathcal K$ is strictly mixing, so is  $\text{L}^2(\widetilde M)\ominus\text{L}^2(M)$ by  \cite[Proposition 7]{OOT15}. Second, if $\Omega=1\in\text{L}^2(M)\subset\mathcal F_M(\mathcal K)$ is the vacuum vector, then $(xs(\xi_n)-s(\xi_n)x)\Omega=x\xi_n-\xi_n x$, for every $n$, and thus $$\text{$\lim_n\|x s(\xi_n)-s(\xi_n)x\|_2=\lim_n\|(xs(\xi_n)-s(\xi_n)x)\Omega\|=\lim_n\|x\xi_n-\xi_n x\|=0$, for every $x\in M$.}$$
Since $\xi_n$ is subtracial, $\langle\xi_n,\xi_n\rangle_M\leq 1$ and so $\|s(\xi_n)\|\leq 2$, for every $n$. Thus, $(s(\xi_n))\in M'\cap\widetilde{M}^{\mathcal U}$. 
Since $\text{E}_M (s(\xi_n)) = 0$ and $\lim_n ||s(\xi_n)||_2 =\lim_n ||s(\xi_n)\Omega|| = \lim_n\|\xi_n\|=1$, we derive that $(s(\xi_n)) \notin M^{\mathcal U}$.

Conversely, assume that there exists a separable II$_1$ factor $\widetilde M\supset M$ such that the $M$-bimodule $\text{L}^2(\widetilde M)\ominus\text{L}^2(M)$ is strictly mixing  and  $M'\cap \widetilde M^{\mathcal U}\not\subset M^{\mathcal U}$. Let $y=(y_n)\in (M'\cap\widetilde M^{\mathcal U})\setminus M^{\mathcal U}$. 
Then $z=y-\text{E}_{M^{\mathcal U}}(y)\not=0$ and $z=(z_n)$, where $z_n=y_n-\text{E}_M(y_n)\in\text{L}^2(\widetilde M)\ominus\text{L}^2(M)$. Since $z\in M'\cap \widetilde M^{\mathcal U}$, we get that $\lim_{n\rightarrow\mathcal U}\|xz_n-z_nx\|_2=0$, for every $x\in M$. Since $M$ is  a factor, and $\text{E}_M(zz^*)\in M'\cap M$, we get that $\text{E}_M(zz^*)=\tau(zz^*)1$. Thus, for every $x\in M$ we have that $$\lim_{n\rightarrow\mathcal U}\langle xz_n,z_n\rangle=\lim_{n\rightarrow\mathcal U}\tau(xz_nz_n^*)=\tau(xzz^*)=\tau(x\text{E}_M(zz^*))=\tau(x)\tau(zz^*).$$ It follows that if we let $\zeta_n=\tau(zz^*)^{-\frac{1}{2}}z_n\in\text{L}^2(\widetilde M)\ominus\text{L}^2(M)$, then $\lim_{n\rightarrow\mathcal U}\|x\zeta_n-\zeta_n x\|_2=0$ and $\lim_{n\rightarrow\mathcal U}\langle x\zeta_n,\zeta_n\rangle=\tau(x)$, for every $x\in M$. This implies that $\text{L}^2(M)\subset_{\text{weak}}\text{L}^2(\widetilde M)\ominus\text{L}^2(M)$. Since the $M$-bimodule $\text{L}^2(\widetilde M)\ominus\text{L}^2(M)$ is also strictly mixing, 
applying \cite[Theorem 9]{OOT15} or \cite[Theorem 3.4]{BF07} gives that $M$ has Haagerup's property. \hfill$\square$ 



\end{document}